\documentclass[a4paper,11pt]{article}

\usepackage[T1]{fontenc}
\usepackage{lmodern}

\usepackage{dsfont}

\usepackage[latin1]{inputenc}
\usepackage{amsmath}
\usepackage{amsthm}
\usepackage{amssymb}
\usepackage{mathrsfs}
\usepackage{graphicx}
\usepackage[all]{xy}
\usepackage{hyperref}

\usepackage{makeidx}

\usepackage{stmaryrd}
\usepackage{caption}

\usepackage{abstract} 

\newtheorem{thm}{Theorem}[section]
\newtheorem{cor}[thm]{Corollary}
\newtheorem{claim}[thm]{Claim}

\newtheorem{lemma}[thm]{Lemma}
\newtheorem{prop}[thm]{Proposition}

\theoremstyle{definition}
\newtheorem{definition}[thm]{Definition}

\newtheorem{remark}[thm]{Remark}
\newtheorem{question}[thm]{Question}

\def\rquotient#1#2{%
	\makeatletter
	\raise.3ex\hbox{$#1$}/\lower.3ex\hbox{$#2$}%
	\makeatother
}	

\makeatletter
\newcommand{\subjclass}[2][2010]{%
	\let\@oldtitle\@title%
	\gdef\@title{\@oldtitle\footnotetext{#1 \emph{Mathematics subject classification.} #2}}%
}
\newcommand{\keywords}[1]{%
	\let\@@oldtitle\@title%
	\gdef\@title{\@@oldtitle\footnotetext{\emph{Key words and phrases.} #1.}}%
}
\makeatother

\newcommand{\Address}{{
		\bigskip
		\small
		
		\textsc{University of Montpellier\\ 
Institut Math\'ematiques Alexander Grothendieck\\
Place Eug\`ene Bataillon\\
34090 Montpellier (France)}\par\nopagebreak
		\textit{E-mail address}: \texttt{anthony.genevois@umontpellier.fr}
		
}}

\makeindex

\title{Rotation groups, mediangle graphs, and periagroups: a unified point of view on Coxeter groups and graph products of groups}
\date{\today}
\author{Anthony Genevois}


\begin{document}

\maketitle

\begin{abstract}
In this article, we introduce \emph{rotation groups} as a common generalisation of Coxeter groups and graph products of groups. We characterise algebraically these groups by presentations (\emph{periagroups}) and we propose a combinatorial geometry (\emph{mediangle graphs}) to study them. As an application, we give natural and unified proofs for several results that hold for both Coxeter groups and graph products of groups. 
\end{abstract}

\vspace{-0.5cm}

\tableofcontents

\section{Introduction}

\noindent
Coxeter groups are famously characterised as \emph{reflection groups}, namely as groups acting on some spaces and generated by \emph{reflections}, i.e. isometries fixing pointwise a codimension-one subspace and flipping the two halfspaces it delimits. More formally, and following \cite[Section~3.2]{Davis}, a group $G$ is a reflection group if there exist a subset $R \subset G$ and a graph $\Omega$ on which it acts such that:
\begin{itemize}
	\item $R$ is stable under conjugation and it generates $G$;
	\item every edge of $\Omega$ is \emph{flipped} by a unique element of $R$ and every element of $R$ flips an edge of $\Omega$;
	\item for every $r \in R$, removing all the edges flipped by $r$ separates the endpoints of any edge flipped by $r$. 
\end{itemize}
Given such a \emph{reflection system}, it can be proved that $\Omega$ coincides with the Cayley graph $\mathrm{Cayl}(G,R_o)$, where $R_o \subset R$ is the set of the elements of $R$ flipping the edges that contain a given basepoint $o \in \Omega$, and that $G$ is a Coxeter group with $R_o$ as a \emph{basis}. Conversely, if $C$ is a Coxeter group with $S$ as a basis, then $(C \curvearrowright \mathrm{Cayl}(C,S), \{ gsg^{-1}, g \in C\})$ defines a reflection system. Moreover, the Cayley graphs appearing here, which we refer to as \emph{Coxeter graphs}, have a rich combinatorial geometry, including a natural \emph{wall structure} (see \cite[Chapters~3 and~4]{Davis}). Most famously, finite Coxeter subgroups naturally define a higher-dimensional cell structure on Coxeter graphs, which can be endowed with a CAT(0) geometry \cite[Chapter~12]{Davis}. Thus, there is a triptych: dynamical behaviour (reflection systems), algebraic structure (Coxeter groups), combinatorial geometry (Coxeter graphs).

\medskip \noindent
A particular case of interest of this picture is given by \emph{right-angled Coxeter groups}. For them, the combinatorial geometry has been axiomatised. A graph $X$ is \emph{median} if every triple of vertices $x_1,x_2,x_3$ admits a unique \emph{median point}, namely a vertex $m \in X$ such that $d(x_i,x_j)=d(x_i,m)+d(m,x_j)$ for all $i \neq j$. Initially introduced in metric graph theory, median graphs have been introduced in geometric group theory through CAT(0) cube complexes and are now well-established in the field. Even though median geometry appears in the study of various groups, the geometry is tightly related to right-angled Coxeter groups. Namely, right-angled Coxeter graphs are median, and, conversely, every median graph embeds as a convex subgraph in a right-angled Coxeter graph \cite{MR2377497}. 

\medskip \noindent
Algebraically, right-angled Coxeter groups are naturally generalised by \emph{graph products of groups} \cite{GreenGP}. Given a graph $\Gamma$ and a collection of groups $\mathcal{G}= \{G_u \mid u \in V(\Gamma)\}$ indexed by the vertices of $\Gamma$, referred to as the \emph{vertex-groups}, the graph product $\Gamma \mathcal{G}$ is the group admitting
$$\langle G_u, \ u \in V(\Gamma) \mid [G_u,G_v]=1, \ \{u,v\} \in E(\Gamma) \rangle$$
as a relative presentation, where $[G_u,G_v]=1$ is a shorthand for: $[a,b]=1$ for all $a \in G_u$ and $b \in G_v$. Usually, one says that graph products of groups interpolate between direct sums (when $\Gamma$ is a complete graph) and free products (when $\Gamma$ has no edge). The key observation for us is that right-angled Coxeter groups coincide with graph products of cyclic groups of order two. Geometrically, we extended the median geometry of right-angled Coxeter groups in \cite{QM} and proposed \emph{quasi-median graphs} as an axiomatisation of the geometry of graph products of groups. This point of view has been extremely useful, as it allowed us to prove combination theorems, e.g. a graph product of finitely many CAT(0) groups is again CAT(0) \cite[Theorem~8.20]{QM}; to characterise relatively hyperbolic graph products \cite[Theorem~8.35]{QM}; to construct quasi-isometrically rigid subgroups in graph products of finite groups such as right-angled Coxeter groups \cite{Excentric}; to re-interpret and generalise Kim and Koberda's results on embeddings between right-angled Artin groups \cite[Theorem~8.43]{QM}; to revisit and generalise Haglund and Wise's theory of special groups \cite{SpecialBis}; and to study the structure of automorphism groups of graph products of groups \cite{MR4295519, AutGP}. As before, even though quasi-median geometry can be applied to various groups (such as Thompson-like groups \cite{QM, MR4033512}, wreath products \cite{QM, Wreath}, and some graphs of groups \cite{QM, SpecialBis}), quasi-median graphs are tightly connected to graph products. Namely, Cayley graphs of graph products of groups (with respect to their vertex-groups) are quasi-median, and, conversely, every quasi-median graph can be realised as a \emph{gated} subgraph in such a Cayley graph \cite{SpecialBis}. 

\medskip \noindent
In order to complete our triptych, though, a dynamical characterisation is missing. Actually, there is an underlying dynamical point of view in the references mentioned above. Because of the central role played by \emph{rotative-stabilisers} of \emph{hyperplanes}, the action of a graph product on its quasi-median graph can be informally thought of as a \emph{right-angled rotation system}, where vertex-groups fix pointwise codimension-one subspaces and permute freely the (possibly infinitely many) connected components they delimit. However, this point of view has not been formalised so far, and it leads to the natural question of the existence of a good notion of \emph{rotation groups}, generalising both Coxeter groups (thought of as reflection groups) and graph products of groups (thought of as right-angled rotation groups). 

\medskip \noindent
In this article, we propose a triptych solving this problem: rotation groups (dynamical behaviour), \emph{periagroups} (algebraic structure), \emph{mediangle graphs} (combinatorial geometry). 

\begin{thm}\label{thm:BigIntro}
If $(G \curvearrowright X, \mathcal{R})$ is a rotation system and $o \in X$ a basepoint, then $G$ is a periagroup with $\mathcal{R}_o$ as a basis and $X$ is a mediangle graph isomorphic to $\mathrm{Cayl}(G,\mathcal{R}_0)$. Conversely, if $G$ is a periagroup with $\mathcal{B}$ as a basis, then $Y:= \mathrm{Cayl}(G, \mathcal{B})$ is a mediangle graph and $(G\curvearrowright Y, \{ gBg^{-1}, g \in G, B \in \mathcal{B}\})$ is a rotation system. 
\end{thm}

\noindent
Here, $\mathcal{R}_o$ denotes the set of the rotative-stabilisers of the cliques containing $o$. In order for our theorem to make sense, let us describe the panels of our triptych.

\paragraph{Rotation groups.} Rotation groups are defined by following the definition of reflection groups given above and inspired by the dynamics of graph products of groups acting on their quasi-median graphs.

\begin{definition}
A \emph{rotation system} is the data of a group $G$, a collection of subgroups $\mathcal{R}$, and a connected graph $\Omega$ on which $G$ acts such that the following conditions are satisfied:
\begin{itemize}
	\item $\mathcal{R}$ is stable under conjugation and it generates $G$.
	\item For every clique $C \subset \Omega$, there exists a unique $R \in \mathcal{R}$ such that $R$ acts freely-transitively on the vertices of $C$. One refers to $R$ as the \emph{rotative-stabiliser} of~$C$.
	\item Every subgroup in $\mathcal{R}$ is the rotative-stabiliser of a clique of $\Omega$.
	\item For every clique $C \subset \Omega$, removing the edges of all the cliques having the same rotative-stabiliser as $C$ separates any two vertices in $C$.
\end{itemize}
\end{definition}

\noindent
Here, a \emph{clique} refers to a maximal complete subgraph. A group $G$ is a \emph{rotation group} if there exist a collection of subgroups $\mathcal{R}$ and a graph $\Omega$ on which $G$ acts such that $(G \curvearrowright \Omega,\mathcal{R})$ is a rotation system. As desired, examples include Coxeter groups and graph products of groups.

\paragraph{Periagroups.} Algebraically, one introduces \emph{periagroups}\footnote{Contraction between ``groups'' and the Greek word ``$\pi \epsilon \rho \iota \alpha \gamma \omega \gamma \eta$'', which can be translated as ``to put in rotative motion''.} as a common generalisation of Coxeter groups and graph products of groups.

\begin{definition}\label{def:Periagroups}
Let $\Gamma$ be a (simplicial) graph, $\lambda : E(\Gamma) \to \mathbb{N}_{\geq 2}$ a labelling of its edges, and $\mathcal{G}= \{G_u \mid u \in V(\Gamma)\}$ a collection of groups indexed by the vertices of $\Gamma$, referred to as the \emph{vertex-groups}. Assume that, for every edge $\{u,v\} \in E(\Gamma)$, if $\lambda(\{u,v\})>2$ then $G_u,G_v$ have order two. The \emph{periagroup} $\Pi(\Gamma,\lambda,\mathcal{G})$ admits
$$\left\langle G_u \ (u \in V(\Gamma)) \mid \langle G_u,G_v \rangle^{\lambda(u,v)}= \langle G_v,G_u \rangle^{\lambda(u,v)} \ (\{u,v\} \in E(\Gamma)) \right\rangle$$
as a relative presentation. Here, $\langle a,b \rangle^k$ refers to the word obtained from $ababab \cdots$ by keeping only the first $k$ letters; and $\langle G_u,G_v \rangle^k = \langle G_v,G_u \rangle^k$ is a shorthand for: $\langle a, b \rangle^k = \langle b,a \rangle^k$ for all non-trivial $a \in G_u$, $b \in G_v$. 
\end{definition}

\noindent
Periagroups of cyclic groups of order two coincide with Coxeter groups; and, if $\lambda \equiv 2$, all the relations are commutations and one retrieves graph products of groups. Thus, periagroups can be thought of as an interpolation between Coxeter groups and graph products of groups. Periagroups of cyclic groups, a common generalisation of Coxeter groups and right-angled Artin groups, also appear in \cite{MR1076077}, and are studied geometrically in \cite{Soergel} under the name \emph{Dyer groups}. 

\medskip \noindent
The vertex-groups given by a fixed description of a group as periagroup define a \emph{basis}. More precisely, let $G$ be a group and $\mathcal{B}$ a collection of non-trivial subgroups. Let $\Gamma$ denote the graph whose vertex-set is $\mathcal{B}$ and whose edges connect two subgroups $A,B \in \mathcal{B}$ whenever $\langle A,B \rangle \neq A \ast B$. For every vertex $u \in \Gamma$, let $G_u$ denote the corresponding subgroup in $\mathcal{B}$, and set $\mathcal{G}=\{G_u, u \in \Lambda\}$. We say that $G$ is a periagroup with $\mathcal{B}$ as a \emph{basis} if there exists a labelling $\lambda : E(\Gamma) \to \mathbb{N}_{\geq 2}$ such that the identity map $\mathcal{B} \to \mathcal{B}$ extends to an isomorphism $\Pi(\Gamma,\lambda,\mathcal{G}) \to G$.

\paragraph{Mediangle graphs.} Our definition of \emph{mediangle graphs} is a variation of the definition of quasi-median graphs as weakly modular graphs with no induced copy of $K_{3,2}$ and $K_4^-$; see Section~\ref{section:Examples}. In order to allow the even cycles arising in the Coxeter graphs, we replace in this definition $4$-cycles with convex even cycles. 

\begin{definition}\label{def:Plurilith}
A connected graph $X$ is \emph{mediangle} if the following conditions are satisfied:
\begin{description}
	\item[(Triangle Condition)] For all vertices $o,x,y \in X$ satisfying $d(o,x)=d(o,y)$ and $d(x,y)=1$, there exists a common neighbour $z \in X$ of $x,y$ such that $z \in I(o,x) \cap I(o,y)$.
	\item[(Intersection of Triangles)] $X$ does not contain an induced copy of $K_4^-$ (i.e. a complete graph $K_4$ minus an edge, or equivalently two $3$-cycles glued along an edge).
	\item[(Cycle Condition)] For all vertices $o ,x,y,z \in X$ satisfying $d(o,x)=d(o,y)=d(o,z)-1$ and $d(x,z)=d(y,z)=1$, there exists a convex cycle of even length that contains the edges $[z,x],[z,y]$ and such that the vertex opposite to $z$ belongs to $I(o,x) \cap I(o,y)$.
	\item[(Intersection of Even Cycles)] The intersection between any two convex cycles of even lengths contains at most one edge. 
\end{description}
\end{definition}

\noindent
Here, given two vertices $x,y \in X$, $I(x,y)$ refers to the \emph{interval} $\{z \in X \mid d(x,y)=d(x,z)+d(z,y)\}$. 

\medskip \noindent
As desired, mediangle graphs extend both Coxeter graphs and quasi-median graphs. A key geometric property is that mediangle graphs are naturally endowed with a wall structure. Namely, define a \emph{hyperplane} as an equivalence class of edges with respect to the transitive closure of the relation that identifies two edges whenever they belong to a common $3$-cycle or when they are opposite in a convex even cycle. It turns out that hyperplanes are as nice as in Coxeter graphs and quasi-median graphs:

\begin{thm}\label{thm:IntroHyp}
In a mediangle graph, a hyperplane separates into at least two convex components. Moreover, the distance between any two vertices coincides with the number of hyperplanes separating them.
\end{thm}

\noindent
We refer to Theorem~\ref{thm:BigHyp} for a more detailed statement. Extending the intuition coming from Coxeter graphs, it is possible to define an \emph{angle} $\measuredangle(J_1,J_2)$ between two \emph{transverse} hyperplanes $J_1,J_2$ in a mediangle graph; see Section~\ref{section:Angles}.

\paragraph{Discussion about Theorem~\ref{thm:BigIntro}.} Now that our main theorem makes sense, a few comments are necessary. 

\medskip \noindent
First of all, in the presentation of a periagroup given by Definition~\ref{def:Periagroups}, there are only commutations between \emph{rotations} (i.e. elements in vertex-groups) that are not of order two. This contrasts with reflections that may generate dihedral groups. Geometrically, this is justified by the fact that, in a mediangle graph, a convex cycle of even length $>4$ cannot share an edge with a $3$-cycle, see Proposition~\ref{prop:InterCycleTriangle}; or equivalently that the angle between two transverse hyperplanes is automatically $\pi/2$ if at least one of these hyperplanes is \emph{thick} (i.e. contains a clique with at least three vertices). In periagroups, one can interpret the latter assertion by saying that the angle between two intersecting axes of rotation is necessarily $\pi/2$ when at least one of the rotations has order $>2$. 

\medskip \noindent
Next, observe that one recovers the intuition suggested in the introduction about graph products of groups. In the right-angled case, i.e. when the angle between any two transverse hyperplane is $\pi/2$, periagroups coincide with graph products of groups and mediangle graphs coincide with quasi-median graphs. So one gets a formal justification of thinking of graph products of groups as right-angled rotation groups. 

\medskip \noindent
Interestingly, our formalism suggests an axiomatisation of the combinatorial geometry of Coxeter graphs, which has not been done so far up to our knowledge: bipartite mediangle graphs, which amounts to considering graphs that are bipartite and that satisfy the third and fourth items from Definition~\ref{def:Plurilith}. However, the geometry captured by mediangle graphs might be larger than the geometry of periagroups. Indeed, contrary to median and quasi-median graphs which can be realised as gated subgraphs of Cayley graphs of right-angled Coxeter groups and graph products of groups, a mediangle graph may not embed as a convex/gated subgraph of the Cayley graph of a periagroup; see Section~\ref{section:Conclusion}. 

\medskip \noindent
Even though one retrieves most of the combinatorial geometry of Coxeter and quasi-median graphs, one can ask whether our axiomatisation is optimal. In other words, are there other axioms needed to get the full strength of this geometry? We suspect that it is not the case, at least for a large part of the geometry.

\paragraph{Applications.} In order to motivate the relevancy of the mediangle geometry as a unified point of view on Coxeter groups and graph products of groups, we deduce from our formalism unified and simple proofs of two results known in several particular cases. 

\begin{thm}\label{thm:IntroRotationSub}
In the periagroup $\Pi(\Gamma,\lambda,\mathcal{G})$, a subgroup generated by conjugates of vertex-groups is a periagroup with vertex-groups that belong to $\mathcal{G}$. 
\end{thm}

\noindent
Among Coxeter groups, this amounts to saying that a subgroup generated by reflections is again a Coxeter group, a classical theorem proved independently in \cite{MR1023969} and \cite{MR1076077}. In fact, \cite{MR1076077} proves the result for periagroups of cyclic groups. Among graph products of groups, one retrieves and improves the consequence \cite[Theorem~10.54]{QM} that shows that a subgroup generated by conjugates of vertex-groups is again a graph product of groups. The same statement for right-angled Artin groups (i.e. graph products of infinite cyclic groups) can also be found in \cite{RAAGRACG}. Our proof is based on a purely geometric statement, namely Theorem~\ref{thm:PingPong}, which is of independent interest and which extends \cite[Theorem~10.54]{QM} (proved for quasi-median graphs). 

\medskip \noindent
For our next statement, we need to define \emph{parabolic subgroups}. Given a decorated graph $\Gamma=\Gamma(\Lambda,\lambda)$ and a collection of groups $\mathcal{G}$ indexed by $V(\Lambda)$, a subgroup $H$ of $\Pi(\Gamma,\mathcal{G})$ is \emph{parabolic} if there exists a subgraph $\Xi \leq \Gamma$ such that $H$ is conjugate to the subgroup generated by the elements and vertex-groups associated to the vertices of $\Xi$, denoted by~$\langle \Xi \rangle$. 

\begin{thm}\label{thm:IntroParabolic}
In a periagroup, the intersection between two parabolic subgroups is again a parabolic subgroup. 
\end{thm}

\noindent
For Coxeter groups, the same result can be found in \cite{MR2466021}; for right-angled Artin groups, in \cite{MR2371978}; and in graph products of groups, in \cite{MR3365774}. Recently, parabolic subgroups received some attention among Artin groups. Theorem~\ref{thm:IntroParabolic} is known for several classes of Artin groups, but it remains open in full generality and it is considered as a difficult problem. The fact that we are able to prove Theorem~\ref{thm:IntroParabolic} for periagroups in a unified way is a good indication that our mediangle geometry is the good geometry.

\paragraph{Organisation of the article.} We begin by fixing the basic definitions and notations in Section~\ref{section:Notation}. The geometry of mediangle graphs is studied in Section~\ref{section:Plurilith}, in which Theorem~\ref{thm:IntroHyp} is proved. Examples of mediangle graphs are also given in Section~\ref{section:Examples}, including one-skeleta of some small cancellation complexes. Next, in Section~\ref{section:Rotation}, we study rotation groups and prove one implication of Theorem~\ref{thm:BigIntro} by showing that they are periagroups. Section~\ref{section:Periagroups} is dedicated to periagroups. We prove that their actions on their Cayley graphs define rotation systems, concluding the proof of Theorem~\ref{thm:BigIntro}; and we deduce from their mediangle geometry a normal form extending the normal forms for Coxeter groups and graph products of groups. Theorems~\ref{thm:IntroRotationSub} and~\ref{thm:IntroParabolic} are proved in the next section. Finally, we include in Section~\ref{section:Conclusion} some concluding remarks and open questions.

\paragraph{Acknowledgements.} I am grateful to Victor Chepoi for his suggestion of (hyper)cellular graphs as a source of examples of mediangle graphs, which can now be found in Section~\ref{section:Examples}.

\section{Basis definitions about graphs}\label{section:Notation}

\noindent
For us, a graph $X$ is a set endowed with an adjacency relation. Consequently, a point in $X$ is automatically a vertex, the edges have distinct endpoints, and two distinct edges cannot have the same endpoints. A path $\alpha$ in $X$ is a sequence of vertices such that any two successive vertices are adjacent. If $x,y \in \alpha$ are two vertices, we denote by $\alpha_{x,y}$ the subpath of $\alpha$ between $x$ and $y$. A path may be identified with its image if this does not create ambiguity.

\medskip \noindent
Given two vertices $x,y \in X$, the \emph{interval} between $x$ and $y$ is
$$I(x,y):= \{ z \in X \mid d(x,y)=d(x,z)+d(z,y)\}.$$
Equivalently, $I(x,y)$ coincides with the union of all the geodesics between $x$ and $y$.

\medskip \noindent
A subgraph $Y \subset X$ is \emph{convex} if a geodesic between any two vertices of $Y$ always lies in $Y$. In other words, $I(x,y) \subset Y$ for all $x,y \in Y$. 

\medskip \noindent
A subgraph $Y \subset X$ is \emph{gated} if, for every $x \in X$, there some $y \in Y$ - referred to as the \emph{gate} or \emph{projection} of $x$ - such that $y \in I(x,z)$ for every $z \in Y$. It is worth noticing that the gate of $x$ coincides with the unique vertex of $Y$ that minimises the distance to $x$. Moreover, gated subgraphs are convex, so one can think about gatedness as a strong convexity property.

\section{Mediangle graphs}\label{section:Plurilith}

\noindent
This section is dedicated to the combinatorial geometry of mediangle graphs as defined by Definition~\ref{def:Plurilith}. It is worth noticing that the second item of the definition provides to \emph{cliques}, i.e. maximal complete subgraphs, a central role due to the following elementary observation.

\begin{lemma}\label{lem:InterClique}
Let $X$ be a graph with no induced copy of $K_4^-$. The intersection between any two distinct cliques contains at most a vertex. 
\end{lemma}

\begin{proof}
If $C_1,C_2 \subset X$ are two cliques whose intersection contains an edge, the no $K_4^-$ condition implies that $C_1 \subset C_2$ and similarly that $C_2 \subset C_1$, hence $C_1=C_2$. 
\end{proof}

\subsection{Convexity}

\noindent
We begin by studying geodesics in mediangle graphs. The main result of this subsection is Proposition~\ref{prop:ConvexCriterion}, which characterises convexity locally. (See Proposition~\ref{prop:LocallyGated} for a similar characterisation of gated subgraphs.) In its statement, given a mediangle graph $X$, we say that a subgraph $Y \subset X$ is \emph{locally convex} if every convex even cycle $\gamma$ that contains a subpath of length $\geq \mathrm{lg}(\gamma)/2$ in $Y$ must lie entirely in $Y$.

\begin{prop}\label{prop:ConvexCriterion}
Let $X$ be a mediangle graph. A subgraph $Y \subset X$ is convex if and only if it is connected and locally convex.
\end{prop}

\noindent
In a graph $X$, a geodesic $\gamma \subset X$ can be turned into a new geodesic by applying a \emph{flip}. Namely, if there exists a convex even cycle $\alpha$ such that $\alpha \cap \gamma$ contains a subsegment $\alpha_0$ of length $\mathrm{lg}(\alpha)/2$, then replacing the subpath $\alpha_0$ in $\gamma$ by $\alpha \backslash \alpha_0$ yields the geodesic we are interested~in. 

\begin{lemma}\label{lem:FlipGeod}
Let $X$ be a graph satisfying the cycle condition and $\alpha, \beta \subset X$ two geodesics with the same endpoints. There exists a sequence of geodesics $\gamma_0=\alpha, \gamma_1, \ldots, \gamma_{n-1}, \gamma_n= \beta$ such that, for every $0 \leq i \leq n-1$, $\gamma_{i+1}$ is obtained from $\gamma_i$ by applying a flip.
\end{lemma}

\begin{proof}
We argue by induction over the length of our geodesics. Fixing one common endpoint $x$ of $\alpha,\beta$, let $e,f$ denote the first edges of $\alpha,\beta$. If $e=f$, then one can apply our induction assumption to $\alpha\backslash e$ and $\beta \backslash f$. 

\medskip \noindent
\begin{minipage}{0.5\linewidth}
Otherwise, if $e \neq f$, we know from the cycle condition that these edges span a convex even cycle $C$ such that, following the notation from the figure on the right, $z \in I(y,a) \cap I(y,b)$. Fix an arbitrary geodesic $[z,y]$ between $z$ and $y$. By applying our induction assumption to $e \cup \alpha_{a,y}$ and $e \cup C_{a,z} \cup [z,y]$, to $f \cup \beta_{b,y}$ and $f \cup C_{b,z} \cup [z,y]$, and by observing that $e \cup C_{a,z} \cup [z,y]$ can be obtained from $f \cup C_{b,z} \cup [z,y]$ by applying a flip, the desired conclusion holds.
\end{minipage}
\begin{minipage}{0.48\linewidth}
\begin{center}
\includegraphics[width=0.8\linewidth]{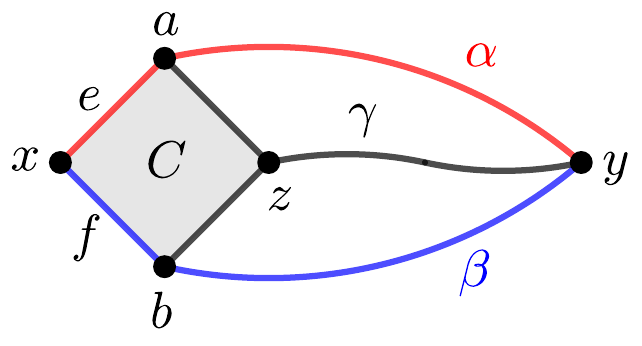}
\end{center}
\end{minipage}

\end{proof}

\noindent
Given a graph $X$, a path $\alpha = (x_0, \ldots, x_n) \subset X$ may be shortened in two ways. If there exists some $0 \leq i \leq n-2$ such that $x_i=x_{i+2}$, then one \emph{removes a backtrack} by replacing $\alpha$ with $(x_0, \ldots, x_i,x_{i+3}, \ldots, x_n)$. If there exists some $0 \leq i \leq n-2$ such that $x_i$ and $x_{i+2}$ are adjacent, one \emph{shortens a triangle} by replacing $\alpha$ with $(x_0, \ldots, x_i,x_{i+2}, \ldots, x_n)$. 

\begin{lemma}\label{lem:ShorteningPath}
Let $X$ be a graph satisfying the triangle and cycle conditions and $x,y \in X$ two vertices. An arbitrary path between $x$ and $y$ can be turned into a geodesic by applying flips, by removing backtracks, and by shortening triangles.
\end{lemma}

\begin{proof}
We argue by induction on the length of our path $\alpha$ between $x$ and $y$. Let $b \in \alpha$ denote the farthest vertex from $x$ such that the subsegment of $\alpha$ connecting $x$ to $b$ is a geodesic. Let $a$ (resp. $c$) denote its predecessor (resp. its successor) along $\alpha$. Two cases may happen, depending whether $d(x,c)=d(x,b)$ or $d(x,c)=d(x,b)-1$.
\begin{center}
\includegraphics[width=\linewidth]{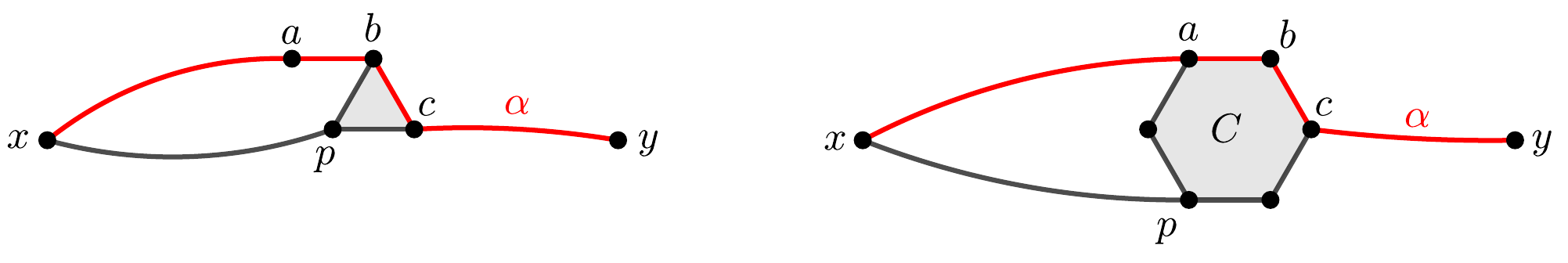}
\end{center}
In the first case, let $p \in X$ denote the vertex obtained by applying the triangle condition. By fixing an arbitrary geodesic $[x,p]$ between $x$ and $p$, we know from Lemma~\ref{lem:FlipGeod} that $[x,p] \cup [p,c] \cup \alpha_{c,y}$ can be obtained from $\alpha$ by applying flips. Next, $[x,p] \cup [p,c] \cup \alpha_{c,y}$ can be shortened by shortening the triangle $(b,c,p)$. The conclusion follows by induction. In the second case, let $p \in X$ denote the vertex obtained by applying the cycle condition. By fixing an arbitrary geodesic $[x,p]$ between $x$ and $p$, we know from Lemma~\ref{lem:FlipGeod} that $[x,p] \cup C_{p,c} \cup C_{c,b} \cup \alpha_{b,y}$ can be obtained from $\alpha$ by applying flips. Next, one can remove a backtrack along the edge $[b,c]$ and obtain the shorter path $[x,p] \cup C_{p,c} \cup \alpha_{c,y}$. Again, the conclusion follows by induction.
\end{proof}

\begin{proof}[Proof of Proposition~\ref{prop:ConvexCriterion}.]
Let $a,b \in Y$ be two vertices. Because $Y$ is connected, there exists a path $\alpha \subset Y$ between $a$ and $b$. According to Lemmas~\ref{lem:FlipGeod} and~\ref{lem:ShorteningPath}, any geodesic between $a$ and $b$ can be obtained from applying flips, removing backtracks, and shortening triangles. All these operations leave a path of $Y$ in $Y$. We conclude that $Y$ is convex. The converse is clear. 
\end{proof}

\subsection{Long convex cycles}

\noindent
This subsection is dedicated to the proof of the following key property satisfied by mediangle graphs:

\begin{thm}\label{thm:LongCyclesConvex}
In a mediangle graph, a convex cycle of even length $>4$ is gated.
\end{thm}

\noindent
We begin by an elementary observation about $4$-cycles.

\begin{lemma}\label{lem:StronglyParallelCliques}
Let $X$ be a graph with no induced copy of $K_4^-$ that satisfies the triangle condition. If there exists a $4$-cycle having two opposite edges in two distinct cliques $C_1$ and $C_2$, then the subgraph of $X$ spanned by $C_1 \cup C_2$ decomposes as the product of $C_1$ with an edge.
\end{lemma}

\begin{proof}
First, notice that no vertex of $C_1$ is adjacent to two vertices of $C_2$. Otherwise, Lemma~\ref{lem:InterClique} implies that $C_1$ and $C_2$ intersect. Let $p \in C_1 \cap C_2$ be a vertex in this intersection. Either $p$ belongs to our $4$-cycle, and we deduce from Lemma~\ref{lem:InterClique} that $C_1=C_2$; or $p$ is adjacent to all the vertices of our $4$-cycle, and we deduce again from Lemma~\ref{lem:InterClique} that $C_1=C_2$. Thus, every vertex of $C_1$ is adjacent to at most one vertex of $C_2$. Similarly, a vertex of $C_2$ cannot be adjacent to two vertices of $C_1$. 

\medskip \noindent
Next, we claim that every vertex $p \in C_1$ is adjacent to a vertex of $C_2$. Let $a_1,b_1,a_2,b_2 \in X$ denote the vertices of our $4$-cycle, with $a_1,b_1 \in C_1$, $a_2,b_2 \in C_2$, and $a_1$ (resp. $b_1$) adjacent to $a_2$ (resp. $b_2$). If $p \in \{a_1,b_1\}$, there is nothing to prove. Otherwise, $d(p,a_2)=d(p,b_2)=2$ since $a_2$ and $b_2$ cannot be adjacent to two vertices of $C_1$. By applying the triangle condition, it follows that there exists a common neighbour $q \in X$ of $p,a_2,b_2$. We know from Lemma~\ref{lem:InterClique} that $q$ belongs to $C_2$. 

\medskip \noindent
Thus, we have proved that every vertex of $C_1$ is adjacent to exactly one vertex of $C_2$. Similarly, a vertex of $C_2$ is adjacent to exactly one vertex of $C_1$. The desired conclusion follows.
\end{proof}

\noindent
The following proposition is the core of the proof of Theorem~\ref{thm:LongCyclesConvex}. Observe that its conclusion does not hold for $4$-cycles since the product of a $4$-cycle with a $3$-cycle is a mediangle graph.

\begin{prop}\label{prop:InterCycleTriangle}
In a mediangle graph, the intersection between a convex cycle of even length $>4$ and a $3$-cycle contains at most one vertex.
\end{prop}

\noindent
Our proof of the proposition will rely on the following lemma, which will be also useful in the next subsection on hyperplanes.

\begin{lemma}\label{lem:PreHyp}
Let $X$ be a graph satisfying the cycle condition. Assume that the intersection between any two distinct convex cycles contains at most one edge. Let $(a,b)$ and $(x,y)$ be two pairs of adjacent vertices satisfying $d(a,x)=d(b,y)=d(a,y)-1=d(b,x)-1$. There exist two sequences of vertices
$$r_0=a, r_1, \ldots, r_{k-1},r_k=x \text{ and } s_0=b,s_1, \ldots, s_{k-1},s_k=y$$
lying respectively on geodesics from $a$ to $x$ and $b$ to $y$ such that, for every $0 \leq i \leq k-1$, $[r_i,s_i]$ and $[r_{i+1},s_{i+1}]$ are opposite edges in a convex even cycle.
\end{lemma}

\begin{proof}
We argue by induction over $d(a,x)=d(b,y)$, which we denote by $d$. If $d=0$, there is nothing to prove, so we assume that $d \geq 1$. Let $p$ be the neighbour of $a$ along an arbitrary geodesic from $a$ to $x$. Because $d(y,p)=d=d(b,y)$ and $d(a,y)=d+1$, we can apply the cycle condition and we find convex even cycle $C$ spanned by the edges $[a,b],[a,p]$ such that the opposite vertex of $a$, say $a'$, belongs to $I(a,y)$. See the figure below.

\medskip \noindent
Let $q$ denote the neighbour of $b$ in $C$ distinct from $a$. Observe that $d(q,x) \leq d(q,y)+d(y,x) = d$ and $d(q,x) \geq d(x,b)-d(b,q) = d$, hence $d(q,x)=d=d(x,a)$. So we know from the cycle condition that the edges $[b,a],[b,q]$ span a convex even cycle $C'$ such that the opposite vertex of $b$, say $b'$, belongs to $I(b,x)$. But $C \cap C'$ contains $[b,a]$ and $[b,q]$, so our condition on the intersection of convex cycles implies that $C=C'$. 

\medskip \noindent
\begin{minipage}{0.5\linewidth}
Let $\ell$ denote half the length of $C$. We have $d(b',x)=d(b,x)-d(b,b')= d-\ell +1$ and $d(a',y)=d(a,y)-d(a,a')=d-\ell+1$. We also have $d(b',y)=d(a,y)-d(a,b')=d-\ell+2$ and $d(a',x)=d(b,x)-d(b,a')=d-\ell +2$. Thus, we can apply our induction assumption to the edges $[b',a'],[x,y]$ and the desired conclusion follows.
\end{minipage}
\begin{minipage}{0.48\linewidth}
\begin{center}
\includegraphics[width=0.7\linewidth]{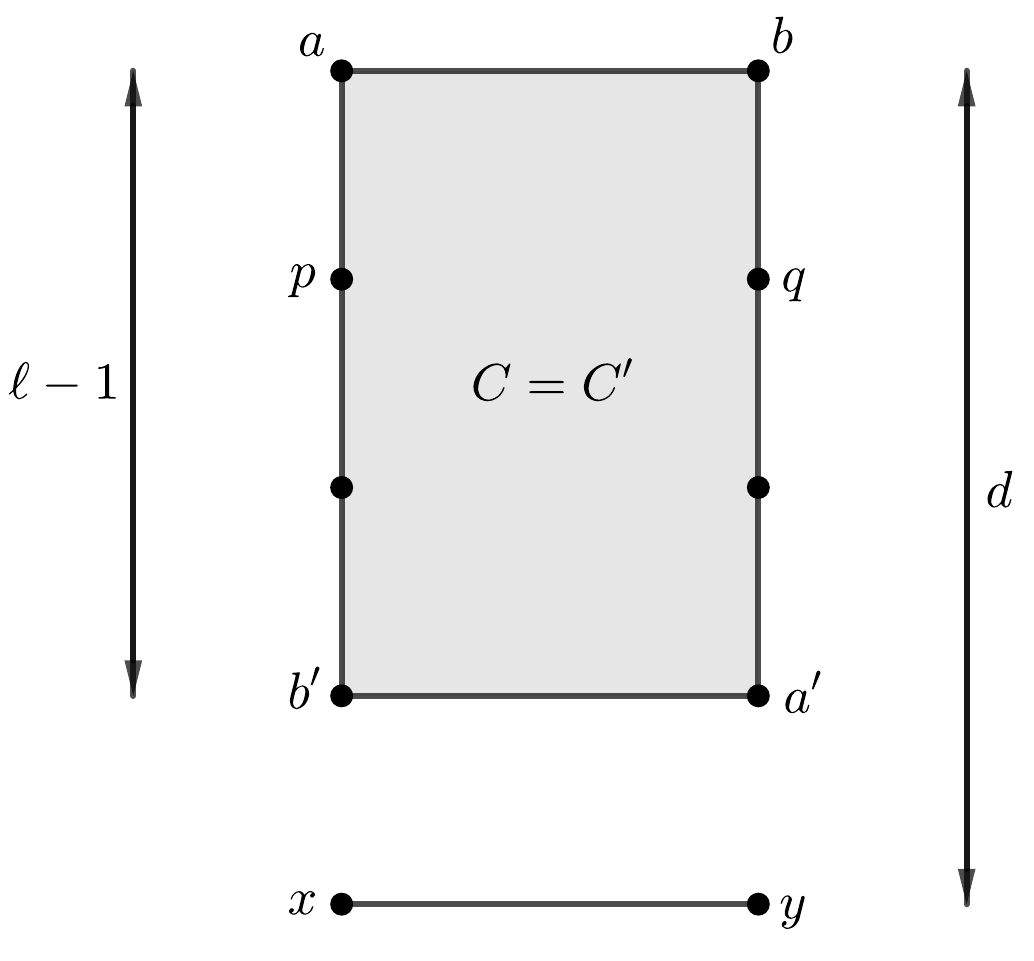}
\end{center}
\end{minipage}
\end{proof}

\begin{proof}[Proof of Proposition~\ref{prop:InterCycleTriangle}.]
Let $X$ be a mediangle graph. Assume for contradiction that there exists a convex cycle $C$ of even length $>4$ sharing an edge with a $3$-cycle. To fix the notation, let $x_0, \ldots, x_{2n-1}$ denote the vertices of $C$ indexed in a cyclic order and let $p \in X$ denote a common neighbour of $x_0$ and $x_1$. Observe that $d(p,x_n) \leq d(p,x_1)+d(x_1,x_n) = n$ and $d(p,x_n) \geq  d(x_n,x_0)-d(x_0,p)=n-1$. But, if $d(p,x_n)=n-1$, then $p$ belongs to $I(x_0,x_n)$, hence $p \in C$ by convexity. Since this is impossible, we must have $d(p,x_n)=n$. One proves similarly that $d(p,x_{n+1})=n$.

\medskip \noindent
\begin{minipage}{0.5\linewidth}
It follows from the triangle condition that $x_n$ and $x_{n+1}$ have a common neighbour $q$ satisfying $d(p,q)=n-1$. By the same argument as in the previous paragraph, one shows that $d(q,x_0)=n=d(q,x_1)$. Thus, one can apply Lemma~\ref{lem:PreHyp} to the edges $[x_0,p],[x_{n+1},q]$, and obtain two sequences of vertices $r_0, \ldots, r_k$ and $s_0, \ldots, s_k$ respectively from $x_0$ to $x_{n+1}$ and from $p$ to $q$. By convexity of $C$, $r_0, \ldots, r_k \in C$. Because of our condition on the intersections of convex cycles, $r_i$ and $r_{i+1}$ must be adjacent for every $0 \leq i \leq k-1$. Thus, we obtain a sequence of $4$-cycles as illustrated by the figure on the right. We apply the same argument to the edges $[p,x_1],[q,x_n]$.
\end{minipage}
\begin{minipage}{0.48\linewidth}
\begin{center}
\includegraphics[width=0.6\linewidth]{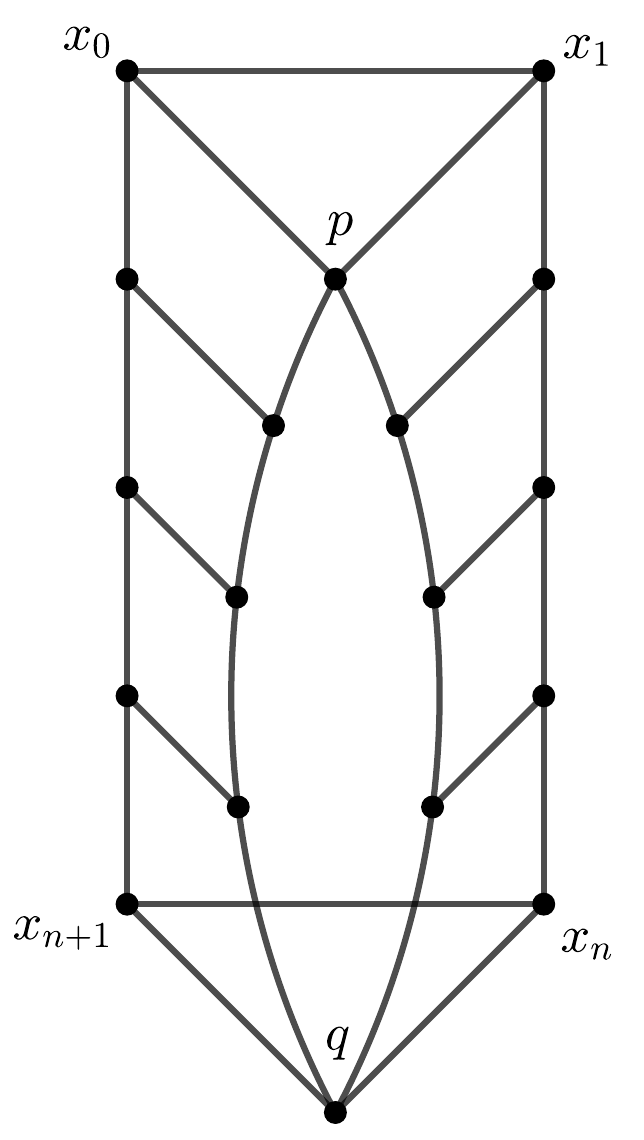}
\end{center}
\end{minipage}

\medskip \noindent
\begin{minipage}{0.48\linewidth}
\begin{center}
\includegraphics[width=0.7\linewidth]{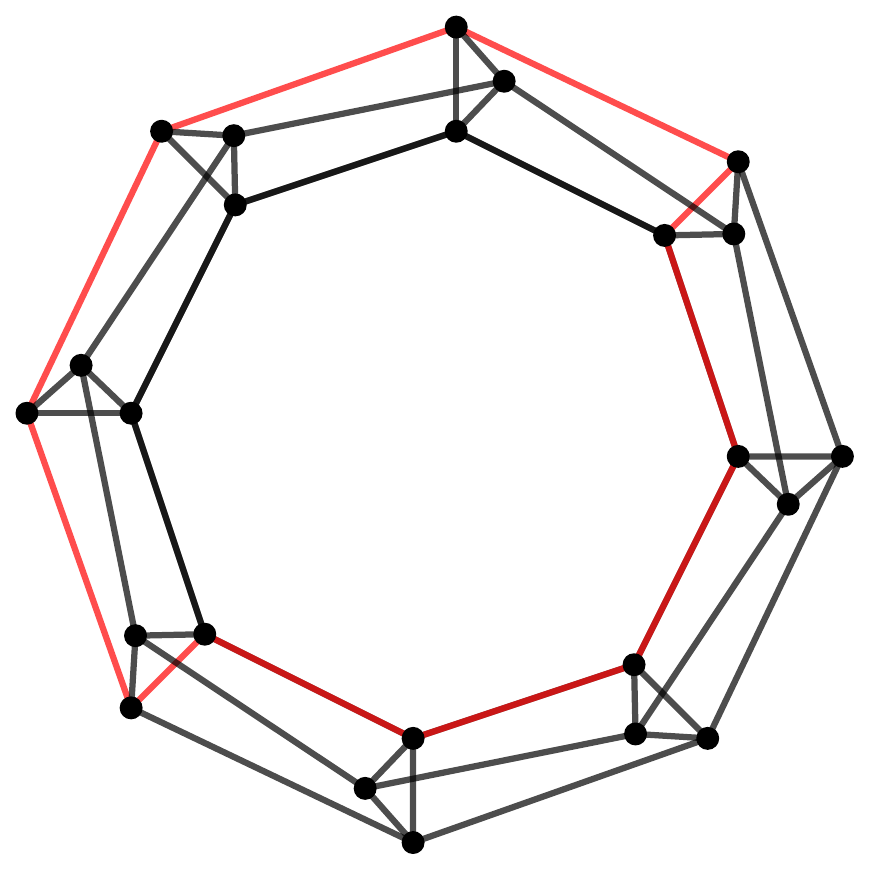}
\end{center}
\end{minipage}
\begin{minipage}{0.5\linewidth}
Now, we apply Lemma~\ref{lem:StronglyParallelCliques} to the cliques containing $3$-cycles $(p,x_0,x_1), (q,x_n,x_{n+1})$ and the edges parallel to $[x_0,p],[p,x_1]$ in all our squares. It follows that our cycle $C$ lies in the product of a shorter cycle with a clique, contradicting the convexity of $C$.
\end{minipage}

\end{proof}

\begin{proof}[Proof of Theorem~\ref{thm:LongCyclesConvex}.]
Let $X$ be a mediangle graph and $C$ is a convex cycle of length $>4$. Fix a vertex $p \in X$ and let $q \in C$ be a vertex minimising the distance to $p$. Denote by $q' \in C$ the vertex opposite to $q$. Let $x_0, \ldots, x_\ell \in C$ denote the vertices of a subpath in $C$ from $q$ to $q'$. We claim that $d(p,x_i)>d(p,x_{i+1})$ for every $0 \leq i \leq \ell-1$. This will prove that $q$ is the gate of $p$ in $C$.

\medskip \noindent
We argue by induction over $i$. We know that $d(p,x_1) \geq d(p,x_0)$, by the definition of $q$; and we know that $d(p,x_1) \neq d(p,x_0)$, since otherwise we would contradict Proposition~\ref{prop:InterCycleTriangle} by applying the triangle condition. So our claim holds for $i=0$. Now, assume that $i \geq 1$. If $d(p,x_{i+1}) \geq d(p,x_i)$, two cases may happen. Either $d(p,x_i)=d(p,x_{i+1})$, but again this would contradict Proposition~\ref{prop:InterCycleTriangle} by applying the triangle condition; or $d(p,x_{i+1})=d(p,x_i)-1$. In the latter case, we have $d(p,x_{i+1})=d(p,x_{i-1})$ because of our induction assumption. The cycle condition implies that the edges $[x_i,x_{i-1}],[x_i,x_{i+1}]$ span a convex even cycle $C'$ such that the opposite vertex of $x_i$, say $x_i'$, belongs to $I(p,x_i)$. But our condition on the intersections of convex cycles implies that $C'=C$. In particular, $x_i'$ belongs to $C$. We have
$$d(p,x_i') = d(p,x_i)- \ell \leq d(p,q)+d(q,x_i)- \ell < d(p,q),$$
contradicting our definition of $q$. Thus, $d(p,x_{i+1})<d(p,x_i)$, as desired.
\end{proof}

\subsection{Hyperplanes}

\noindent
In this subsection, we introduce hyperplanes and prove their main properties. Let us begin by fixing some definitions and notations.

\medskip \noindent
Given a mediangle graph $X$, a \emph{hyperplane} is an equivalence class of edges with respect to the transitive closure of the relation that identifies any two edges that belong to a common $3$-cycle or that are opposite in a convex even cycle. Given a hyperplane $J$, we denote by $X \backslash \backslash J$ the graph with the same vertices as $X$ and whose edges are those of $X$ that do not belong to $J$. The connected components of $X \backslash \backslash J$ are the \emph{sectors delimited by $J$}. The \emph{carrier} $N(J)$ of $J$ is the union of all the convex cycles containing edges in $J$, and the connected components of $N(J) \backslash \backslash J$ are the \emph{fibres}. Two vertices of $X$ are \emph{separated by $J$} if they lie in distinct sectors. Two hyperplanes are \emph{transverse} if they contain two distinct pairs of opposite edges in some convex even cycle.

\medskip \noindent
We sum up the main results of the subsection in the following statement:

\begin{thm}\label{thm:BigHyp}
Let $X$ be a mediangle graph. The following assertions hold.
\begin{itemize}
	\item[(i)] Hyperplanes separate $X$. More precisely, given a hyperplane $J$ and a clique $C \subset J$, any two vertices in $C$ are separated by $J$ and every sector delimited by $J$ contains a vertex of $C$. 
	\item[(ii)] Sectors delimited by hyperplanes are always convex.
	\item[(iii)] A path is a geodesic if and only if it crosses each hyperplane at most once. As a consequence, the distance between any two vertices coincides with the number of hyperplanes separating them.
\end{itemize}
\end{thm}

\begin{remark}
Theorem~\ref{thm:BigHyp} focuses on the results needed in the rest of the article, but more can be said about hyperplanes. For instance:
\begin{itemize}
	\item[(iv)] Carriers and fibres delimited by hyperplanes are always isometrically embedded.
	\item[(v)] Carriers, fibres, and sectors delimited by square-hyperplanes are always gated. Moreover, for every square-hyperplane $J$, for every clique $C \subset J$, and for every fibre $F \subset J$, the map $x \mapsto (\mathrm{proj}_F(x), \mathrm{proj}_C(x))$ defines an isometry $N(J) \to F \times C$.
\end{itemize}
Here, a \emph{square-hyperplane} refers to a hyperplane crossing only $3$- and $4$-cycles. Observe that, contrary to quasi-median graphs, sectors may not be gated. For instance, in the graph given by the hexagonal tiling of the Euclidean plane, no sector is gated.
\end{remark}

\noindent
We begin by proving two elementary observations.

\begin{lemma}\label{lem:4CycleConvex}
In a graph satisfying the cycle condition, every induced $4$-cycle is convex.
\end{lemma}

\begin{proof}
An induced $4$-cycle is isometrically embedded. Therefore, as a consequence of the cycle condition, two opposite vertices belong to a convex $4$-cycle. Its convexity implies that it has to coincide with the original $4$-cycle.
\end{proof}

\noindent
Lemma~\ref{lem:4CycleConvex} is of course false for longer cycles. For instance, $3$-cubes contains isometrically embedded $6$-cycles that are not convex. 

\begin{lemma}\label{lem:CliquesGated}
Let $X$ be a graph satisfying the triangle condition with no $K_4^-$. Cliques in $X$ are gated.
\end{lemma}

\begin{proof}
Let $C \subset X$ be a clique and $p \in X$ a vertex. If there exist two vertices $a,b \in C$ minimising the distance to $p$, then the triangle condition implies that $a,b$ have a common neighbour $c \in X$ satisfying $d(p,c)=d(p,C)-1$. Lemma~\ref{lem:InterClique} implies that $c$ belongs to $C$, a contradiction. Thus, there exists a unique vertex in $C$ minimising the distance to $p$. It defines the gate of $p$ in $C$. 
\end{proof}

\noindent
Given a gated clique $C$ in a graph $X$, we refer to the partition $\{ \mathrm{proj}_C^{-1}(x) \mid x \in C\}$ of $X$ as the \emph{projection-partition associated to $C$}. 

\begin{lemma}\label{lem:ProjectionPartitions}
Let $X$ be a mediangle graph. The projection-partitions defined by two cliques $C_1,C_2$ lying in the same hyperplane coincide. 
\end{lemma}

\begin{proof}
We assume that $C_1$ and $C_2$ contain opposite edges from a convex even cycle $C$, the general case following by iteration. 

\medskip \noindent
First, assume that $C$ has length $>4$. As a consequence of Proposition~\ref{prop:InterCycleTriangle}, $C_1$ and $C_2$ are edges. Because $C$ is gated according to Theorem~\ref{thm:LongCyclesConvex}, the projection on $C_1$ (resp. $C_2$) coincides with the projection on $C$ followed by the projection on $C_1$ (resp. $C_2$). As a consequence, if we set $C_1= \{a_1,b_1\}$ and $C_2= \{a_2,b_2\}$ with $d(a_1,a_2)<d(a_1,b_2)$, then $\mathrm{proj}_{C_1}^{-1}(a_1)=\mathrm{proj}_{C_2}^{-1}(a_2)$ and $\mathrm{proj}_{C_1}^{-1}(b_1)= \mathrm{proj}_{C_2}^{-1} (b_2)$.

\medskip \noindent
Next, assume that $C$ is a $4$-cycle. According to Lemma~\ref{lem:StronglyParallelCliques}, $C_1 \cup C_2$ decomposes as a product of $C_1$ with an edge. In other words, every vertex in $C_1$ (resp. $C_2$) is adjacent to exactly one vertex in $C_2$ (resp. $C_1$). We claim that, for every vertex $p \in X$, if $a_1 \in C_1$ denotes its projection on $C_1$, then its projection on $C_2$ is the unique vertex $a_2 \in C_2$ adjacent to $a_1$. This is allows us to conclude since this observation implies that 
$\mathrm{proj}_{C_1}^{-1}(x)= \mathrm{proj}_{C_2}^{-1}(\sigma(x))$ where $\sigma : C_1 \to C_2$ denotes the map that sends each vertex of $C_1$ to its unique neighbour in $C_2$.

\medskip \noindent
Assume for contradiction that the projection $b_2 \in C_2$ of $p$ is distinct from $a_2$. So $d(p,a_2)=d(p,b_2)+1$. For convenience, we set $k:=d(p,a_1)$. Observe that
$$k+1 =d(p,a_1)+d(a_1,b_1)= d(p,b_1) \leq d(p,b_2)+d(b_2,b_1) = d(p,b_2)+1,$$
hence $d(p,b_2) \geq k$, or equivalently $d(p,a_2) \geq k+1$. We also have $|d(p,a_2)-k| \leq 1$. Consequently, $d(p,a_2)=k+1$ and $d(p,b_2)=k$. 

\medskip \noindent
\begin{minipage}{0.5\linewidth}
By applying the cycle condition, we know that the edges $[a_2,b_2]$ and $[a_2,a_1]$ span a convex even cycle whose vertex opposite to $a_2$ belongs to $I(p,a_2)$. But it follows from Lemma~\ref{lem:4CycleConvex} and from our condition on the intersections on convex cycles that the $4$-cycle $(a_1,b_1,b_2,a_2)$ is our convex even cycle. Consequently, $d(p,b_1)=d(p,a_2)-2=k-1<d(p,a_1)$, a contradiction with the definition of $a_1$.
\end{minipage}
\begin{minipage}{0.48\linewidth}
\begin{center}
\includegraphics[width=0.7\linewidth]{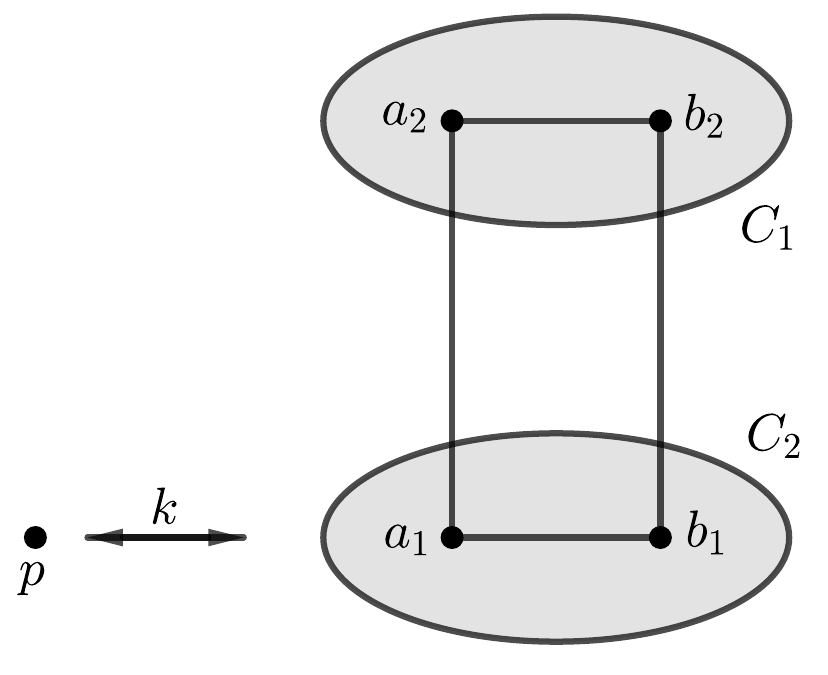}
\end{center}
\end{minipage}

\end{proof}

\begin{proof}[Proof of Theorem~\ref{thm:BigHyp}.]
We begin by proving the following observation:

\begin{claim}\label{claim:HyperplanesSeparate}
Let $J$ be a hyperplane and $x,y \in X$ two adjacent vertices. If $[x,y]$ belongs to $J$, then $J$ separates $x$ and $y$. 
\end{claim}

\noindent
According to Lemma~\ref{lem:ShorteningPath}, every path from $x$ to $y$ can be transformed into the edge $[x,y]$ by applying flips, by removing backtracks, and by shortening triangles. Consequently, in order to conclude that every path from $x$ to $y$ crosses $J$, it suffices to observe that, if the result of such an elementary operation crosses $J$, then the initial path has to cross $J$, which is clear. Claim~\ref{claim:HyperplanesSeparate} is proved. 

\medskip \noindent
Let $J$ be a hyperplane and $C \subset J$ a clique. Let us prove that the connected components of $X \backslash \backslash J$ coincide with pieces of projection-decomposition given by $C$. This will prove $(i)$ since any two vertices in $C$ obviously have distinct projections on $C$, and because the vertices along an arbitrary geodesic from one vertex of $X$ to its projection on $C$ all have the same projection on $C$. In other words, we want to prove that:

\begin{claim}\label{claim:HypCliqueSameProj}
Two vertices of $X$ lie in the same sector delimited by $J$ if and only if they have the same projection on $C$. 
\end{claim}

\noindent
We begin by proving the claim for adjacent vertices. If two adjacent vertices $x,y \in X$ have distinct projections on $C$, then it follows from Lemma~\ref{lem:PreHyp} that the edge $[x,y]$ belongs to $J$, and we deduce from Claim~\ref{claim:HyperplanesSeparate} that $x,y$ lie in distinct sectors delimited by $J$. Conversely, if $x,y$ lie in distinct sectors, then necessarily the edge $[x,y]$ belongs to $J$ and it follows from Lemma~\ref{lem:ProjectionPartitions} that $x,y$ have distinct projections on $C$ (since they obviously have distinct projections on the unique clique containing $[x,y]$). 

\medskip \noindent
From now on, we do not assume anymore that our two vertices $x,y$ are adjacent. If they have the same projection on $C$, say $p$, then the union of two geodesics from $x$ to $p$ and from $p$ to $y$ defines a path $\alpha$ from $x$ to $y$ all of whose edges have their endpoints with the same projection on $C$. It follows from the case previously studied that any two consecutive vertices along $\alpha$ lie in the same sector delimited by $J$, proving that $x$ and $y$ also lie in the same sector. Conversely, if $x,y$ have distinct projections on $C$, then along any path between $x$ and $y$ there will be at least one edge whose endpoints have distinct projections on $C$; from the case previously studied, this implies that this edge belongs to $J$. Thus, every path between $x$ and $y$ has to contain an edge from $J$, proving that $J$ separates $x$ and $y$. This concludes the proof of Claim~\ref{claim:HypCliqueSameProj}. 

\medskip \noindent
So far, we have proved $(i)$. Now, let us prove $(ii)$. So fix a hyperplane $J$ and $x,y \in X$ two vertices not separated by $J$. We first claim that there exists a geodesic between $x$ and $y$ not crossed by $J$. To see this, let $\gamma$ be an arbitrary geodesic from $x$ to $y$. If $J$ does not cross $\gamma$, we are done. Otherwise, let $e,f \subset \gamma$ denote the first and last edges lying in $J$. Necessarily, the first and last endpoint $a,b$ of $e,f$ respectively belong to the same fibre $F$ of $J$, namely the fibre contained in the sector delimited by $J$ that contains $x,y$. According to Lemma~\ref{lem:PreHyp}, which applies to $e,f$ thanks to Claim~\ref{claim:HypCliqueSameProj}, there exists a geodesic $\alpha$ from $e$ to $f$ lying in $F$. Because two edges of a convex even cycle lie in the same hyperplane if and only if they are opposite, as a consequence of Claim~\ref{claim:HypCliqueSameProj}, $J$ does not cross $\alpha$. Thus, by replacing the subpath of $\gamma$ between $a,b$ by $\alpha$, one obtains a geodesic between $x,y$ not crossed by $J$. Consequently, it follows from Lemma~\ref{lem:FlipGeod} that, in order to show that no geodesic between $x,y$ is crossed by $J$, it suffices to observe that applying a flip to a path not crossed by $J$ produces a new path that is still not crossed by $J$, which is clear. Thus, $(ii)$ is proved. 

\medskip \noindent
Next, let us prove $(iii)$. We begin by proving that a path crossing a hyperplane twice can be shortened, and so is not a geodesic. Let $J$ be a hyperplane and let $\gamma$ be a path crossing $J$ (at least) twice. Let $[a,b]$ (resp. $[p,q]$) be the first (resp. second) edge along $\gamma$ that lies in $J$. Necessarily, $p$ and $b$ lie in the same fibre of $J$, which amounts to saying that the projection of $p$ on the (unique) clique $C$ containing $[a,b]$ is $b$. Let $c$ denote the projection of $q$ on $C$. Because $[p,q]$ belongs to $J$, we know that $p$ and $q$ have different projections on $C$, so $c \neq b$. If $c=a$, then $a$ and $q$ belong to the same sector delimited by $J$, and it follows from $(ii)$ that $\gamma$ is not a geodesic. So let us assume that $c \neq a$. 

\medskip \noindent
\begin{minipage}{0.5\linewidth}
Observe that
$$d(q,c)<d(q,b) \leq 1+d(p,b),$$
where the strict inequality is justified by the fact that $c$ is the unique vertex of $C$ minimising the distance to $q$. Similarly, $d(p,b) < 1+d(q,c)$, proving that $d(p,b)=d(q,c)$. Thus, replacing the subpath $\gamma_{a,q}$ of $\gamma$ with the concatenation of the edge $[a,c]$ and an arbitrary geodesic $[c,q]$ from $c$ to $q$ shortens~$\gamma$.
\end{minipage}
\begin{minipage}{0.48\linewidth}
\begin{center}
\includegraphics[width=0.7\linewidth]{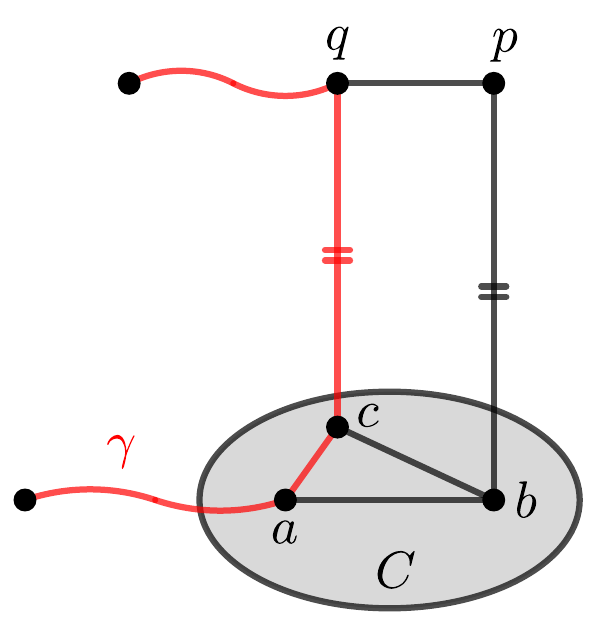}
\end{center}
\end{minipage}

\medskip \noindent
Thus, our claim is proved. Observe that it implies that the distance between two vertices $x,y \in X$ coincides with the number of hyperplanes $N(x,y)$ separating them. Indeed, the inequality $d(x,y) \geq N(x,y)$ is clear since a path from $x$ to $y$ must cross all the hyperplanes separating $x,y$. But now we know that the hyperplanes crossed by a given geodesic from $x$ to $y$ are crossed only once, which implies that they all separate $x,y$. Therefore, the length of this geodesic, which is $d(x,y)$, must be at least $N(x,y)$. Hence $d(x,y)=N(x,y)$ as desired. As a consequence, a consequence, a path is a geodesic if and only if its length coincides with the number of hyperplanes separating its endpoints, which amounts to saying that it crosses each hyperplane at most once. The item $(iii)$ is proved.
\end{proof}

\subsection{Projections}

\noindent
In this section, we collect a few results describing the intersection between hyperplanes and projections on gated subgraphs. 

\begin{prop}\label{prop:MinVertices}
Let $X$ be a mediangle graph and $Y,Z \subset X$ two gated subgraphs. If $y \in Y$ and $z \in Z$ are two vertices minimising the distance between $Y$ and $Z$, then the hyperplanes separating $y$ and $z$ are exactly the hyperplanes separating $Y$ and $Z$.
\end{prop}

\noindent
We begin by proving the proposition when one of the two gated subgraphs is a single vertex.

\begin{lemma}\label{lem:ProjHypSep}
Let $X$ be a mediangle graph, $x \in X$ a vertex, and $Y \subset X$ a gated subgraph. The hyperplanes separating $x$ from its projection on $Y$ coincide with the hyperplanes separating $x$ from $Y$. 
\end{lemma}

\begin{proof}
Let $y \in Y$ denote the projection of $x$ on $Y$ and let $J$ be a hyperplane crossing $Y$. Fix a vertex $z \in Y$ such that $J$ separates $y$ and $z$. Because $y$ is the projection of $x$, it lies on a geodesic from $x$ to $z$. According to Theorem~\ref{thm:BigHyp}.$(iii)$, such a geodesic cannot cross $J$ twice, so its subsegment from $x$ to $y$ does not cross $J$. A fortiori, $J$ does not separate $x$ from $y$. Thus, a hyperplane separating $x$ from $y$ does not cross $Y$, which implies that it separates $x$ from $Y$. Conversely, it is clear that a hyperplane separating $x$ from $Y$ separates $x$ from $y$. 
\end{proof}

\begin{cor}\label{cor:ProjLip}
Let $X$ be a mediangle graph, $x,y \in X$ two vertices, and $Y \subset X$ a gated subgraph. The hyperplanes separating the projections of $x,y$ on $Y$ coincide with the hyperplanes that cross $Y$ and separate $x,y$. As a consequence, the projection on $Y$ is $1$-Lipschitz.
\end{cor}

\begin{proof}
If a hyperplane separates the projections $x',y'$ of $x,y$, then it must cross $Y$, so it follows from Lemma~\ref{lem:ProjHypSep} that it cannot separate $x$ from $x'$ nor $y$ from $y'$. Necessarily, it has to separate $x$ and $y$. Conversely, if a hyperplane separates $x,y$ and crosses $Y$, then it follows again from Lemma~\ref{lem:ProjHypSep} that it cannot separate $x$ from $x'$ nor $y$ from $y'$, which implies that it has to separate $x'$ and $y'$. This proves the first assertion of our corollary. The second assertion follows from what we have just proved and Theorem~\ref{thm:BigHyp}.$(iii)$. 
\end{proof}

\begin{proof}[Proof of Proposition~\ref{prop:MinVertices}.]
Clearly, the hyperplanes separating $Y$ and $Z$ separate $y$ and $z$. Conversely, since $y$ has to be the projection of $z$ on $Y$ and $z$ the projection $y$ on $Z$, applying Lemma~\ref{lem:ProjHypSep} twice implies that a hyperplane separating $y$ and $z$ cannot cross $Y$ nor $Z$, so it has to separate $Y$ and $Z$. 
\end{proof}

\noindent
We conclude the subsection with the following observation, which will be useful later.

\begin{lemma}\label{lem:ProjProj}
Let $X$ be a mediangle graph and $Y,Z \subset X$ two gated subgraphs. For every $z \in \mathrm{proj}_Z(Y)$, $\mathrm{proj}_Z \left( \mathrm{proj}_Y(z) \right)=z$. As a consequence, $z$ and $\mathrm{proj}_Y(z)$ are two vertices minimising the distance between $Y$ and $Z$. 
\end{lemma}

\begin{proof}
Let $p \in Y$ be a vertex such that $z= \mathrm{proj}_Z(p)$. Let $z'$ denote the projection of $z$ on $Y$ and $z''$ the projection of $z'$ on $Z$. If $z \neq z''$, then there exists some hyperplane $J$ separating $z$ and $z''$. Because $z= \mathrm{proj}_Z(p)$ and $z''= \mathrm{proj}_Z(z')$, it follows from Corollary~\ref{cor:ProjLip} that $J$ separates $p$ and $z'$. A fortiori, it crosses $Y$. But $J$ also separate $z$ and $z'= \mathrm{proj}_Y(z)$, contradicting Lemma~\ref{lem:ProjHypSep}. 
\end{proof}

\subsection{Angles between hyperplanes}\label{section:Angles}

\noindent
Given two transverse hyperplanes $J_1,J_2$ in a mediangle graph $X$, we want to define an angle $\measuredangle(J_1,J_2)$ between them. First, if they both cross a convex even cycle $C$, we define the \emph{angle between $J_1,J_2$ at $C$} by
$$\measuredangle_C(J_1,J_2): = 2 \pi \cdot \frac{1+d( J_1 \cap C, J_2 \cap C)}{\mathrm{length}(C)}.$$
This number coincides with the geometric angle between the straight lines connecting the midpoints of the two pairs of opposite edges in $J_1,J_2$ when $C$ is thought of as a regular Euclidean polygon. The key observation is that this angle does not depend on the cycle $C$ under consideration, allowing us to define the \emph{angle} $\measuredangle(J_1,J_2)$ as the angle at an arbitrary convex even cycle they both cross.

\begin{prop}\label{prop:Angle}
Let $X$ be a mediangle graph and $J_1,J_2$ two transverse hyperplanes. We have
$$\measuredangle_{C_1}(J_1,J_2)= \measuredangle_{C_2}(J_1,J_2)$$
for any two convex even cycles $C_1,C_2$ both crossed by $J_1,J_2$. 
\end{prop}

\noindent
The proposition will be an easy consequence of the following result, which is of independent interest:

\begin{prop}\label{prop:ProjCycle}
Let $X$ be a mediangle graph and $C_1,C_2 \subset X$ two convex cycles of length $>4$. The projection of $C_1$ on $C_2$ is either a single vertex, a single edge, or $C_2$ entirely. In the latter case, the projection on $C_2$ restricts to an isometry $C_1 \to C_2$. 
\end{prop}

\begin{proof}
Assume that the projection of $C_1$ on $C_2$ is neither a single vertex nor a single edge. Because this projection has to be connected, as a consequence of Corollary~\ref{cor:ProjLip}, this amounts to saying that $\mathrm{proj}_{C_2}(C_1)$ contains at least three consecutive vertices, say $a_2,b_2,c_2$. According to Lemma~\ref{lem:ProjProj}, $a_2,b_2,c_2$ are the respective projections of $a_1:= \mathrm{proj}_{C_1}(a_2)$, $b_1:= \mathrm{proj}_{C_1}(b_2)$, $c_1:= \mathrm{proj}_{C_1}(c_2)$, which are also successive vertices along $C_1$. Let $p_1$ (resp. $p_2$) denote the vertex of $C_1$ (resp. $C_2$) opposite to $b_1$) (resp. $b_2$). 

\medskip \noindent
\begin{minipage}{0.5\linewidth}
Set $D:= d(C_1,C_2)$ and let $\ell_1$ (resp. $\ell_2$) denote half the length of $C_1$ (resp. $C_2$). According to Lemma~\ref{lem:ProjProj} and Proposition~\ref{prop:MinVertices}, the distances $d(a_1,a_2)$, $d(b_1,b_2)$, and $d(c_1,c_2)$ are all equal to $D$.  
\end{minipage}
\begin{minipage}{0.48\linewidth}
\begin{center}
\includegraphics[width=0.9\linewidth]{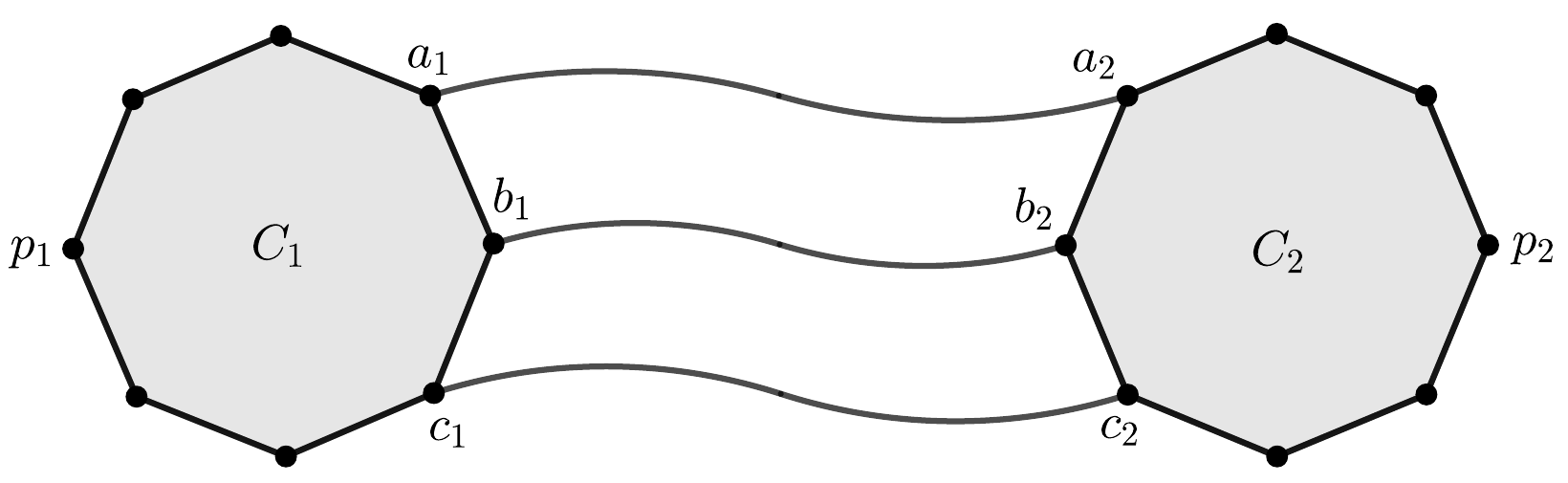}
\end{center}
\end{minipage}

\medskip \noindent
Observe that $d(p_1,b_2)=D+\ell_1$ and $d(p_1,a_2)=D+\ell_1-1=d(p_1,c_2$. Therefore, we can apply the cycle condition and find a convex cycle spanned by the edges $[b_2,a_2],[b_2,c_2]$ such that the vertex opposite to $b_2$ belongs to $I(p_1,b_2)$. But our condition on the intersections of convex cycles implies that this cycle must be $C_2$, hence $p_2 \in I(p_1,b_2)$. Similarly, one shows that $p_1 \in I(p_2,b_1)$. Thus, we have
$$D+ \ell_1= d(p_1,b_2)=d(p_1,p_2)+d(p_2,b_2)=d(p_1,p_2)+\ell_2$$
and 
$$D+ \ell_2= d(p_2,b_1)= d(p_1,p_2)+d(p_1,b_1)= d(p_1,p_2)+ \ell _1.$$
By combining these two equalities, we deduce that $\ell_1=\ell_2$ and that $d(p_1,p_2)=D$. It follows from the definition of $D$ that $p_2$ is the projection of $p_1$ on $C_2$. It follows that $\mathrm{proj}_{C_2}(C_1)=C_2$. Indeed, the shortest path along $C_1$ between $p_1$ and $c_1$ has to send to a path of length $< \ell_2$ between $p_2$ and $c_2$, so it must be the shortest path along $C_2$ between $p_2$ and $c_2$. Similarly, the shortest path along $C_2$ between $p_2$ and $a_2$ must lie in $\mathrm{proj}_{C_2}(C_1)$.

\medskip \noindent
Thus, we have proved that the projection on $C_2$ restricts to a surjective map $C_1 \to C_2$. In fact, since we have seen that $C_1$ and $C_2$ have the same size, the map is a bijection, its inverse being given by the projection on $C_1$. Since projections are $1$-Lipschitz, we conclude that $C_1 \to C_2$ is an isometry. 
\end{proof}

\begin{proof}[Proof of Proposition~\ref{prop:Angle}.]
It follows from Corollary~\ref{cor:ProjLip} that $\mathrm{proj}_{C_2}$ sends an edge in $J_1 \cap C_1$ (resp. $J_2 \cap C_1$) to an edge of $J_1 \cap C_2$ (resp. $J_2 \cap C_2$). As a consequence, the projection of $C_1$ on $C_2$ contains at least two edges, so Proposition~\ref{prop:ProjCycle} implies that $\mathrm{proj}_{C_2}$ induces an isometry $C_1 \to C_2$. As said before, the projection sends edges from $J_1$ (resp. $J_2$) to edges from $J_1$ (resp. $J_2$). Since the angles between $J_1$ and $J_2$ at $C_1$ and $C_2$ are defined metrically, the desired conclusion follows. 
\end{proof}

\subsection{Examples of mediangle graphs}\label{section:Examples}

\noindent
In this section, we collect a few examples of mediangle graphs.

\paragraph{Median and quasi-graphs.} A connected graph $X$ is \emph{median} if, for all vertices $x,y,z \in X$, the intersection $I(x,y) \cap I(y,z) \cap I(x,z)$ is reduced to a single vertex, referred to as the \emph{median point}. One can think of median graphs as one-skeleta of nonpositively curved cell complexes made of cubes \cite{MR1748966, Roller}, such as products of simplicial trees. 

\begin{prop}\label{prop:Mmediangle}
A graph is median if and only if it is mediangle and it has no $3$-cycle and no convex even cycle of length $>4$. In particular, median graphs are mediangle.
\end{prop}

\noindent
A connected graph $X$ is \emph{weakly modular} if the following conditions are satisfied:
\begin{description}
	\item[(Triangle Condition)] For all vertices $o,x,y \in X$ satisfying $d(o,x)=d(o,y)$ and $d(x,y)=1$, there exists a common neighbour $z \in X$ of $x,y$ such that $z \in I(o,x) \cap I(o,y)$.
	\item[(Quadrangle Condition)] For all vertices $o,x,y,z \in X$ satisfying $d(o,x)=d(o,y)=d(o,z)-1$ and $d(x,z)=d(y,z)=1$, there exists a common neighbour $w \in X$ of $x,y$ such that $w \in I(o,x) \cap I(o,y)$.
\end{description}
If in addition $X$ does not contain an induced subgraph isomorphic to $K_4^-$ (i.e. a complete graph $K_4$ minus an edge, or equivalently two $3$-cycles glued along an edge) or to the bipartite complete graph $K_{3,2}$ (i.e. two $4$-cycles glued along two consecutive edges), then $X$ is a \emph{quasi-median graph}. 

\medskip \noindent
One can think of quasi-median graphs as a generalisation of median graphs that allows $3$-cycles. Indeed, it turns out that median graphs coincide with quasi-median graphs with no $3$-cycles. Following this idea, one can also think of quasi-median graphs as one-skeleta of nonpositively curved cell complexes made of \emph{prisms}, i.e. products of simplices \cite{QM}. 

\begin{prop}\label{prop:QMmediangle}
A graph is quasi-median if and only if it is mediangle and has no convex even cycle of length $>4$. In particular, quasi-median graphs are mediangle.
\end{prop}

\begin{proof}
The proposition follows from the fact that the convex even cycles in a quasi-median graph $X$ are the induced $4$-cycles. Let $C \subset X$ be a convex even cycle. Fix a vertex $o \in C$, let $o' \in C$ denote the vertex opposite to it, and let $a,b \in C$ denote the neighbours of $o'$. Then the quadrangle condition provides a vertex $o'' \in X$ adjacent to both $a,b$ that belongs to $I(o,a) \cap I(o,b) \subset C$. This implies that $C$ must have length $4$. Conversely, an induced $4$-cycle is automatically isometrically embedded, so, if $X$ contains an induced $4$-cycle that is not convex, then there must be an induced copy of $K_{3,2}$ in $X$, which is impossible. 
\end{proof}

\begin{proof}[Proof of Proposition~\ref{prop:Mmediangle}.]
The proposition is a direct consequence of Proposition~\ref{prop:QMmediangle} once we know that median graphs coincide with quasi-median graphs with no $3$-cycles \cite[(25) p. 149]{MR605838} (see also \cite[Proposition~3]{MR2397349} and \cite[Corollary~2.92]{QM}).
\end{proof}

\paragraph{Cayley graphs of Coxeter groups.} Given a graph $\Gamma$ and a labelling $\lambda : E(\Gamma) \to \mathbb{N}_{\geq 3}$ of its edges, the \emph{Coxeter group} $C(\Gamma)$ is defined by the presentation
$$\langle u \in V(\Gamma) \mid u^2=1 \ (u \in V(\Gamma)), \ (uv)^{\lambda(u,v)}=1 \ (\{u,v\} \in E(\Gamma)) \rangle.$$
Its Cayley graph, constructed with respect to its generating set $V(\Gamma)$, is referred to as a \emph{Coxeter graph}.

\begin{prop}\label{prop:CoxeterPlurilith}
Coxeter graphs are mediangle.
\end{prop}

\noindent
The proposition is a particular case of Theorem~\ref{thm:BigIntro}. Nevertheless, we include a short elementary proof.

\begin{proof}[Proof of Proposition~\ref{prop:CoxeterPlurilith}.]
Coxeter graphs are bipartite (as a consequence, for instance, of the reduction and normal form that hold in Coxeter groups \cite{MR0254129} (see also \cite[Theorem~3.4.2]{Davis})), so the triangle condition and the no $K_4^-$ condition hold. If $u,v \in \Gamma$ are adjacent vertices, then
$$1, \ u, \ uv, \ uvu, \ldots, (uv)^{\lambda(u,v)}=1$$
defines a $2\lambda(u,v)$-cycle in the Coxeter graph $X(\Gamma)$ of $C(\Gamma)$, which we refer to as a \emph{dihedral cycle}. Dihedral cycles are convex as a consequence of the normal form that holds in Coxeter groups \cite{MR0254129} (see also \cite[Theorem~3.4.2]{Davis}). Observe that, as a consequence of \cite[Lemma~4.7.3.$(iii)$]{Davis}, $X(\Gamma)$ satisfies the cycle condition with respect to its dihedral cycles. More precisely, if $[1,r]$ and $[1,s]$ are two distinct edges such that $r,s \in I(o,g)$ for some $g \in X(\Gamma)$, then $\langle r,s \rangle$ is a finite subgroup, namely a dihedral group of order $2\lambda(r,s)$ whose longest element is $m:=\langle r,s \rangle^{\lambda(r,s)}$, and $d(1,g)=d(1,m)+d(m,g)$. By translating such a configuration by elements of $C(\Gamma)$, one obtains the desired cycle condition for dihedral cycles. From this, it follows easily that every convex cycle must be dihedral. Moreover, because a dihedral cycle is uniquely determined by two consecutive edges, the intersection between any two distinct dihedral cycles cannot contain more than one edge. 
\end{proof}

\paragraph{Small cancellation complexes.} Here we assume that the reader is familiar with classical small cancellation theory in polygonal complexes, and refer to \cite{MR1812024, MR1888425} for more details. Roughly speaking, a simply connected polygonal complex $P$ satisfies the condition $C(p)$ for some integer $p$ if the boundary of a polygon in $P$ cannot be covered by fewer than $p$ other polygons; and it satisfies the condition $T(q)$ for some integer $q$ if the links of its vertices do not contain cycles of length strictly between $2$ and $q$. In view of the restriction on the intersections between long cycles in Definition~\ref{def:Plurilith}, we focus our attention on polygonal complexes that are \emph{even} (i.e. all the polygons have even length) and \emph{simple} (i.e. we assume that the intersection between any two distinct polygons contains at most one edge). This is a strong restriction, but such complexes already provide a lot of interesting examples; see Section~\ref{section:Conclusion} for some of them.

\begin{prop}\label{prop:SmallCancellation}
One-skeleta of simple, even, simply connected polygonal complexes that satisfy the small cancellation conditions C(4)-T(4) are bipartite mediangle graphs.
\end{prop}

\begin{proof}
Let $X$ be a simple, even, simply connected polygonal complex satisfying C(4)-T(4). Because $X$ is simply connected and even, its one-skeleton $X^{(1)}$ is necessarily bipartite. Consequently, the triangle condition and the no $K_4^-$ condition are automatic. The classification of disc diagrams in C(4)-T(4) complexes (see for instance \cite[Theorem~9.4]{MR1888425}) easily shows that boundaries of polygons are convex in $X^{(1)}$ and that a $4$-cycle always bound a polygon (using the fact that $X$ is simple). Now we claim that $X^{(1)}$ satisfies the cycle condition with respect to the boundaries of the polygons in $X$. This will imply the cycle condition and the fact that the convex cycles of $X^{(1)}$ are exactly the boundaries of the polygons in $X$. Since $X$ is simple, the intersection between any two distinct convex cycles cannot contain more than one edge, concluding the proof of the proposition.

\medskip \noindent
Assume for contradiction that our claim fails. Let $o,x,y,z \in X$ be the distinct vertices with $d(o,x)=d(o,y)=d(o,z)-1$ and $d(z,x)=d(z,y)=1$ in our configuration that contradicts our cycle condition. Let $D \to X$ be a disc diagram bounded by $\alpha \cup [x,z] \cup [z,y] \cup \beta$ where $\alpha,\beta$ are geodesics from $o$ to $x$ and from $o$ to $y$. We choose $o,x,y,z$ and $\alpha,\beta$ in order to minimise the area of $D$ among all the counterexamples of our cycle condition. As a consequence, $\alpha,\beta$ cannot meet at an interior point, which implies that $D$ is non-degenerate (i.e. homeomorphic to a disc). Observe that $D$ cannot be a single polygon nor a ladder. As a consequence of the classification of disc diagrams in C(4)-T(4) complexes (see for instance \cite[Theorem~9.4]{MR1888425}), $D$ has to contains at least three $i$-shells with $i \leq 2$. Because $\alpha,\beta$ are geodesics, the outer path of such an $i$-shell cannot lie entirely on $\alpha$ or $\beta$, unless the $i$-shell is a square. But, in the latter case, we would be able to decrease the are of $D$ by modifying $\alpha$ or $\beta$. Therefore, there must be an $i$-shell containing $[x,z]$ and the last edge of $\alpha$ in its outer path, as well as an $i$-shell containing $[y,z]$ and the last edge of $\beta$ in its outer path. Such $i$-shells must be squares since $\alpha \cup [x,z]$ and $\beta \cup [y,z]$ are geodesics. But then, by flipping one of these two squares, one finds a new counterexample to our cycle condition with a disc diagram of smaller area, a contradiction.
\end{proof}

\begin{remark}
It is worth noticing that C(6)-T(3) complexes may not produced mediangle graphs. For instance, a wheel of three $4$-cycles (i.e. the one-skeleton of a $3$-cube minus a vertex) is the one-skeleton of a square complex satisfying the small cancellation condition C(6)-T(3) but it is not mediangle.
\end{remark}

\paragraph{Hypercellular graphs.} Introduced in \cite{MR4065838} as generalising median and cellular graphs \cite{MR1379365}, hypercellular graphs can be defined as isometrically embedded subgraphs in hypercubes in which all finite convex subgraphs can be obtained from Cartesian products of edges and even cycles by successive gated amalgamations. 

\begin{prop}
Hypercellular graphs are mediangle.
\end{prop}

\begin{proof}
Let $X$ be a hypercellular graph. Because $X$ is bipartite, the triangle condition and the condition on the intersections of triangles from Definition~\ref{def:Plurilith} are empty, and a fortiori satisfied. Now, let $A,B \subset X$ be two convex cycles sharing at least one edge. Let $a$ and $b$ denote the two endpoints of $A \cap B$. Travelling along $A \backslash A \cap B$ from $a$ to $b$, let $x \in A$ denote the last vertex that is closer to $a$ than $b$. So $d(x,b)=d(x,a)+1$. Because the path from $x$ to $a$ in $A$ is a geodesic intersecting $B$ only at $a$, $a$ must be the gate of $x$ in $B$ (according to \cite[Theorem~1]{MR4065838}, convex cycles are gated in hypercellular graphs), hence
$$1+ d(x,a)= d(x,b)=d(x,a)+d(a,b),$$
hence $d(a,b)=1$. In other words, $A \cap B$ is reduced to a single edge. Thus, the condition on the intersections of even cycles from Definition~\ref{def:Plurilith} holds.

\medskip \noindent
Finally, let us verify the cycle condition. So let $o,x,y,z \in X$ be four vertices satisfying $d(o,x)=d(o,y)=d(o,z)-1$, $d(x,z)=d(y,z)=1$, and $d(x,y)=2$. According to \cite[Proposition~8]{MR4065838}, there exists a gated subgraph $P$ and three vertices $o',x',y' \in P$ such that $P$ is isometric to a product of edges and even cycles and
$$\left\{ \begin{array}{l} d(o,x)=d(o,o')+ d(o',x')+ d(x',x) \\ d(o,y)=d(o,o')+d(o',y')+d(y',y) \\ d(x,y)= d(x,x')+d(x',y')+d(y',y) \end{array} \right..$$
Because $z$ cannot belong to $I(o,x) \cup I(o,y)$, necessarily $z$ belongs to $I(x',y') \backslash \{x',y'\}$. But $d(x,y)=2$, so the only possibility is that $x=x'$ and $y=y'$. So $x,y \in P$, also $z \in P$ because $P$ is convex. But the cycle condition holds in a product of edges and even cycles, so it holds in particular for $o',x,y,z$. Necessarily, it also holds for $o,x,y,z$.  
\end{proof}

\section{Rotation groups}\label{section:Rotation}

\subsection{Rotation systems}

\noindent
Generalising \emph{reflection systems} in order to obtain \emph{rotations}, i.e. isometries that fix pointwise codimension-one subspaces and permute freely the connected components they delimit, one naturally gets the following definition:

\begin{definition}
A \emph{rotation presystem} is the data of a group $G$, a collection of subgroups $\mathcal{R}$, and a connected graph $X$ on which $G$ acts such that the following conditions are satisfied:
\begin{itemize}
	\item $\mathcal{R}$ is stable under conjugation and it generates $G$.
	\item For every clique $C \subset X$, there exists a unique $R \in \mathcal{R}$ such that $R$ acts freely-transitively on the vertices of $C$. One refers to $R$ as the \emph{rotative-stabiliser} of~$C$.
	\item Every subgroup in $\mathcal{R}$ is the rotative-stabiliser of a clique of $X$.
\end{itemize}
If the following additional condition is satisfied, $(G \curvearrowright X, \mathcal{R})$ is a \emph{rotation system}:
\begin{itemize}
	\item For every clique $C \subset X$, removing the edges of all the cliques having the same rotative-stabiliser as $C$ separates any two vertices in $C$.
\end{itemize}
\end{definition}

\noindent
In this subsection, we focus on the following statement:

\begin{prop}\label{prop:FreeTransitive}
Let $(G \curvearrowright X, \mathcal{R})$ be rotation system. The action $G \curvearrowright X$ is free and transitive on the vertices.
\end{prop}

\noindent
In fact, the action is already vertex-transitive for rotation presystems:

\begin{lemma}\label{lem:Transitive}
Let $(G \curvearrowright X, \mathcal{R})$ be rotation presystem. The action $G \curvearrowright X$ is transitive on the vertices.
\end{lemma}

\begin{proof}
Let $x,y \in X$ be two vertices. Because $X$ is connected, there exists a path $z_1, \ldots, z_n$ from $x$ to $y$. For every $1 \leq i \leq n-1$, $z_i$ and $z_{i+1}$ are adjacent, so there exists an element $gg_i \in G$ in the rotative-stabiliser of a clique containing both $z_i,z_{i+1}$ such that $g_i \cdot z_i=z_{i+1}$. Then
$$g_{n-1} \cdots g_1 \cdot x = g_{n-1} \cdots g_1 \cdot z_1 = z_n=y,$$
concluding the proof of the lemma.
\end{proof}

\noindent
The following statement will be needed in the proof of Proposition~\ref{prop:FreeTransitive} and in the next section. Observe that the lemma implies that, in the graph, an edge belongs to a unique clique.

\begin{lemma}\label{lem:InterCliqueRotation}
Let $(G \curvearrowright X, \mathcal{R})$ be rotation system. The intersection between any two distinct cliques in $X$ is either empty or reduced to a single vertex.
\end{lemma}

\begin{proof}
Assume for contradiction that there exist two distinct cliques $C_1,C_2 \subset X$ such that $C_1 \cap C_2$ contains an edge $[a,b]$. Fix two vertices $p_1 \in C_1 \backslash C_2$ and $p_2 \in C_2 \backslash C_1$. Let $J$ denote the union of the edges of all the cliques with the same rotative-stabiliser as $C_1$. Because $J$ separates $a$ and $b$, it has to contain at least one of the two edges $[p_2,a]$, $[p_2,b]$. Say that $[p_2,a] \in J$. So there exist two elements $g,h$ in the rotative-stabiliser of $C_1$ satisfying $g \cdot p_1 =a$ and $h \cdot a = p_2$. But then $hg \cdot p_1 = p_2 \notin C_2$, contradicting the fact that $g$ and $h$ stabilise $C_1$. 
\end{proof}

\begin{proof}[Proof of Proposition~\ref{prop:FreeTransitive}.]
We already know from Lemma~\ref{lem:Transitive} that the action is transitive. So let $o \in X$ be a vertex and $g \in G$ an element such that $g \cdot o = o$. Our goal is to show that $g=1$.

\medskip \noindent
Let $S$ denote the union of the rotative-stabilisers of all the cliques containing $o$. We claim that $S$ generates $G$. In order to prove this assertion, it suffices to show that, for every clique $C$, its rotative-stabiliser lies in $\langle  S \rangle$. We argue by induction over $d(o,C)$. If $d(o,C)=0$, then $C$ lies in $S$ by definition of $S$. If $d(o,C) \geq 1$, then let $e \subset X$ denote the first edge of a geodesic $\gamma$ from $o$ to $C$. Let $s \in S$ be an element of the rotative-stabiliser of a clique containing $e$ that sends the second endpoint of $e$ to $o$. Then $s \cdot \gamma\backslash e$ defines a path from $o$ to $s \cdot C$ of length $\mathrm{lg}(\gamma)-1$, hence $d(o,s \cdot C)<d(o,C)$. Our induction assumption implies that the rotative-stabiliser of $sC$ lies in $\langle S \rangle$. Since $s \in S$, so does the rotative-stabiliser of $C$, concluding the proof of our claim.

\medskip \noindent
Write $g$ as a product $s_1 \cdots s_n$ where $s_1, \ldots, s_n \in S \backslash \{1\}$. If $n=0$, then $g=1$ and there is nothing to prove, so we assume that $n \geq 1$. Observe that
$$o, s_1 \cdot o, s_1s_2 \cdot o, \ldots, s_1 \cdots s_n \cdot o=g \cdot o=o$$
defines a loop in $X$. Let $C$ be a clique containing the edge $[o,s_1 \cdot o]$. Because the edges of all the cliques with that same rotative-stabiliser as $C$ separates $X$, there must exist some $2 \leq k \leq n-1$ such that $[s_1 \cdots s_k \cdot o, s_1 \cdots s_{k+1} \cdot o]$ belongs to a clique $C'$ with the same rotative-stabiliser as $C$. On the other hand, we know that $s_{k+1}$ belongs to the rotative-stabiliser of some clique $C''$ containing $o$, so the conjugate $s_{k+1}^{s_1 \cdots s_k}$ belongs to the rotative-stabiliser of the clique $s_1 \cdots s_k \cdot C''$. Since the latter contains $s_1 \cdots _k \cdot o$, it also contains the $s_{k+1}^{s_1 \cdots s_k}$-translate of $s_1 \cdots s_k \cdot o$, namely $s_1 \cdots s_{k+1} \cdot o$. Thus, the cliques $C'$ and $s_1 \cdots s_k \cdot C''$ contain both $s_1 \cdots s_k \cdot o$ and $s_1 \cdots s_{k+1} \cdot o$, which implies, according to Lemma~\ref{lem:InterCliqueRotation}, that the cliques coincide. A fortiori, $s_{k+1}^{s_1 \cdots s_k}$ belongs to the rotative-stabiliser of $C$, so we can write
$$g= s_1 \cdots s_n = s_{k+1}^{s_1 \cdots s_k} \cdot s_1 s_2 \cdots s_{k-1}s_{k+1} \cdots s_n = s_1' s_2 \cdots s_{k-1} s_{k+1} \cdots s_n$$
where $s_1':= s_{k+1}^{s_1 \cdots s_k}s_1$ belongs to $\mathrm{stab}_\circlearrowright(C) \subset S$. Thus, we obtain a new expression of $g$ as a product of elements of $S$, but which is shorter. By iterating the argument, we eventually conclude that $g=1$, as desired.
\end{proof}

\subsection{Mediangle geometry}

\noindent
In this section, we show how rotation systems and mediangle graphs are connected. Namely:

\begin{thm}\label{thm:RotationPlurilith}
For every rotation system $(G \curvearrowright X, \mathcal{R})$, the graph $X$ is mediangle.
\end{thm}

\noindent
We begin by verifying that the triangle condition holds.

\begin{prop}\label{prop:TriangleConditionRotation}
For every rotation system $(G \curvearrowright X, \mathcal{R})$, the graph $X$ satisfies the triangle condition.
\end{prop}

\noindent
In the graph given by a rotation system, we define a \emph{barrier} as the union of the edges of all the cliques with the same rotative-stabiliser as a given clique. The connected components of the graph obtained by removing the edges from a barrier are the \emph{domains} delimited by that barrier. Barriers and domains correspond to hyperplanes and sectors in mediangle graphs. However, we introduce a different terminology in this section because we do not know yet that our graph is mediangle.

\medskip \noindent
We know from the definition of a rotation system that, given a clique, the associated barrier separates any two vertices of the cliques, so there at least as many domains as there are vertices in the clique. The next lemma shows that there are actually no more domains delimited by the barrier.

\begin{lemma}\label{lem:AtMostCC}
Let $(G \curvearrowright X, \mathcal{R})$ be a rotation presystem. For every clique $C \subset X$ and every vertex $x \in X$, there exists some $c \in C$ such that $x$ lies in the same domain as $c$.
\end{lemma}

\begin{proof}
Fix an arbitrary point $p \in C$ and an arbitrary path $\gamma$ from $x$ to $p$. If $J$ does cross $\gamma$, then $p$ lies in the same domain as $p$ and we are done. Otherwise, there exists an edge $[a,b]$ along $\gamma$ lying in a clique with the same rotative-stabiliser as $C$. Let $r \in G$ denote the element of this rotative-stabiliser that sends $b$ to $a$. Observe that $r \cdot p$ still belongs to $C$ since $r$ stabilises $C$. Replacing the subpath $\gamma_{a,p}$ of $\gamma$ with $r \cdot \gamma_{b,p}$ yields a shorter path from $p$ to a vertex in $C$. By iterating the process, we eventually find a path from $x$ to a point of $C$ that does not cross $J$.
\end{proof}

\noindent
Next, we show that geodesics and barriers have the same behaviour as in mediangle graphs.

\begin{lemma}\label{lem:BarrierGeod}
Let $(G \curvearrowright X, \mathcal{R})$ be a rotation presystem. A path in $X$ that crosses twice a barrier can be shortened.
\end{lemma}

\begin{proof}
Let $\gamma$ be a path crossing some barrier $J$ (at least) twice. Let $[a,b]$ and $[p,q]$ be respectively the first and second edges along $\gamma$ that lie in $J$. Let $r \in G$ denote the element of the rotative-stabiliser of some clique from $J$ containing $[a,b]$ that sends $b$ to $a$. Because $r$ stabilises some clique containing $[p,q]$, $r \cdot p$ is at distance $\leq 1$ from $q$. By replacing the subsegment $\gamma_{a,q}$ of $\gamma$ by $r \cdot \gamma_{b,p} \cup [r \cdot p,q]$, one obtains a new path of length $\leq \mathrm{lg}(\gamma)-1$. 
\end{proof}

\begin{cor}\label{cor:BarrierGeod}
Let $(G \curvearrowright X, \mathcal{R})$ be a rotation system. A path in $X$ is a geodesic if and only if it crosses each barrier at most once. As a consequence, the distance between any two vertices in $X$ coincides with the number of barriers separating them.
\end{cor}

\begin{proof}
Of course, the distance between two vertices $x,y \in X$ is at least the number $N(x,y)$ of hyperplanes separating them. On the other hand, because of Lemma~\ref{lem:BarrierGeod} and because we are in a rotation system, the edges of a geodesic yield pairwise distinct hyperplanes that separate its endpoints, hence $d(x,y) \leq N(x,y)$. Thus, $d(x,y)=N(x,y)$. Consequently, a path is a geodesic if and only if it only crosses the hyperplanes separating its endpoints and only once, which amounts to saying that it crosses each hyperplane at most once. 
\end{proof}

\begin{proof}[Proof of Proposition~\ref{prop:TriangleConditionRotation}.]
Let $o,x,y \in X$ be three vertices satisfying $d(o,x)=d(o,y)$ and $d(x,y)=1$. Let $C$ denote the clique containing the edge $[x,y]$ and let $c \in C$ denote the vertex given by Lemma~\ref{lem:AtMostCC}. Because domains are convex as a consequence of Lemma~\ref{lem:BarrierGeod}, there exists a geodesic $\gamma$ from $o$ to $c$ that is not crossed by the barrier $J$ containing $C$. Consequently, the concatenations $\gamma \cup [c,x]$ and $\gamma \cup [c,y]$ do not cross $J$ twice. It follows from Corollary~\ref{cor:BarrierGeod} that these two paths are geodesics, hence $d(o,c)=d(o,x)-1=d(o,y)-1$, as desired.
\end{proof}

\noindent
The next step towards the proof of Theorem~\ref{thm:RotationPlurilith} is to show that the cycle condition holds. In fact, we will show that that the cycle condition holds for a specific family of cycles, and next we will observe that these cycles actually coincide with the convex even cycles.

\begin{definition}
Let $(G \curvearrowright X, \mathcal{R})$ be a rotation system. Fix a vertex $o \in X$ and let $a,b \in G$ be two elements that belong to the rotative-stabilisers of two distinct cliques containing $o$. If there exists some $n \geq 2$ such that $\langle a,b \rangle^{m} = \langle b,a \rangle^m$, then
$$o, a \cdot o, ab \cdot o, \ldots, \langle a,b \rangle^m \cdot o = \langle b,a \rangle^m \cdot o, \ldots, ba \cdot o , b \cdot o, o$$
is a \emph{dihedral cycle} in $X$. 
\end{definition}

\noindent
As announced, we begin by showing the cycle condition holds with respect to dihedral cycles.

\begin{lemma}\label{lem:CycleConditionRotation}
Let $(G \curvearrowright X, \mathcal{R})$ be a rotation system. For all vertices $o,x,y,z \in X$ satisfying $d(o,x)=d(o,z)=d(o,y)-1$ and $d(y,x)=d(y,z)=1$, there exists a dihedral cycle $C$ containing the edges $[y,x],[y,z]$ such that the vertex opposite to $o$ belongs to $I(o,x) \cap I(o,z)$. 
\end{lemma}

\begin{proof}
Let $a \in G$ (resp. $b \in G$) denote the element in the rotative-stabiliser of the clique containing $[y,x]$ (resp. $[x,z]$) satisfying $a \cdot y =x$ (resp. $b \cdot y=z$). Let $A$ (resp. $B$) denote the barrier containing $[y,x]$ (resp. $[y,z]$).

\begin{claim}\label{claim:CCZero}
Let $\gamma$ be a geodesic from $x$ to $o$ passing through $y$ (resp. $z$) and let $\gamma_0$ be a subpath containing $[x,y]$ (resp. $[x,z]$). If $\gamma$ does not cross $B$ (resp. $A$), then $[y,z] \cup b \cdot \gamma_0$ (resp. $[y,x] \cup a \cdot \gamma_0$) lies on a geodesic from $x$ to $o$.
\end{claim}

\noindent
\begin{minipage}{0.5\linewidth}
Because $B$ separates $y$ and $o$, there must be some edge $[p,q] \subset \gamma$ lying in $B$, say with $p$ closer to $y$ than $q$. Since $B$ does not cross $\gamma_0$ by assumption, $p$ does not belong to the interior of $\gamma_0$. Observe that $q=b \cdot p$. Indeed, $b$ stabilises the clique containing $[p,q]$ and it sends the domain containing $y$ (and a fortiori $p$) to the domain containing $z$ (and a fortiori $q$). Thus, $b \cdot \gamma_{y,p}$ is a geodesic from $z$ to $b \cdot q=p$ containing $b \cdot \gamma_0$ that has the same length as $\gamma_{y,p}$. 
\end{minipage}
\begin{minipage}{0.48\linewidth}
\begin{center}
\includegraphics[width=0.7\linewidth]{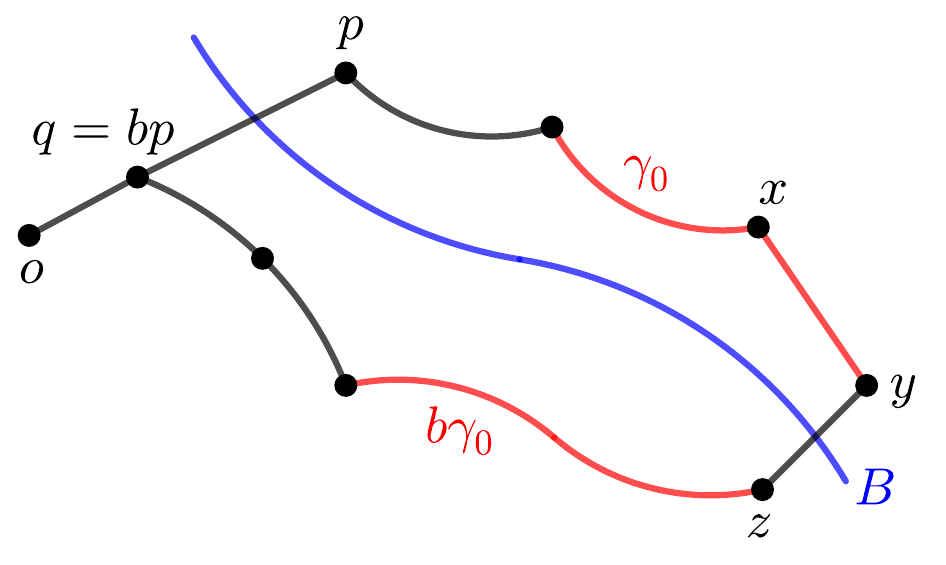}
\end{center}
\end{minipage}

\medskip \noindent
Because
$$\mathrm{lg}([y,z] \cup b \cdot \gamma_{y,p} \cup \gamma_{q,o}) = \mathrm{lg}( \gamma_{q,o}) + \mathrm{lg}(\gamma_{y,p}) + 1 = \mathrm{lg}(\gamma)=d(o,y),$$
we deduced that $[y,z] \cup b \cdot \gamma_0$ lies on the geodesic $[y,z] \cup b \gamma_{y,p} \cup \gamma_{q,o}$, concluding the proof of Claim~\ref{claim:CCZero}. 

\medskip \noindent
Now, we apply Claim~\ref{claim:CCZero} iteratively. For every $k \geq 0$, if $A$ and $B$ do not cross $(y,a \cdot y, ab \cdot y, \ldots, \langle a,b \rangle^k y)$ and $(y,b \cdot y, ba \cdot y, \ldots, \langle b,a \rangle^k \cdot y)$ respectively, then we know that we can extend these two geodesics towards $o$ by adding respectively the vertices $\langle a,b \rangle^{k+1} \cdot y$ and $\langle b,a \rangle^{k+1}$. The process has to stop after at most $d(y,o)$ steps. So there exists a first $k \geq 1$ such that $A$ separates $\langle b,a \rangle^{k-1} \cdot y$ and $\langle b,a \rangle^{k-1} \cdot y$ or $B$ separates $\langle a,b \rangle^{k-1} \cdot y$ and $\langle a,b \rangle^{k-1} \cdot y$. Up to switching the roles played by $a$ and $b$, we assume that we are in the latter case. 

\medskip \noindent
\begin{minipage}{0.5\linewidth}
Because $b$ stabilises the clique containing $\langle a,b \rangle^{k-1} \cdot y$ and $\langle a,b \rangle^{k} \cdot y$, and sends the domain delimited by the corresponding barrier that contains $y$ (a fortiori $\langle a,b \rangle^{k-1} \cdot y$) to the domain containing $z$ (and a fortiori $\langle a,b \rangle^k \cdot y$), we must have $b \langle a,b \rangle^{k-1} \cdot y = \langle a,b \rangle^k \cdot y$, hence $\langle b,a \rangle^k = \langle a,b \rangle^k$ because $G$ acts on $X$ with trivial vertex-stabilisers. 
\end{minipage}
\begin{minipage}{0.48\linewidth}
\begin{center}
\includegraphics[width=0.9\linewidth]{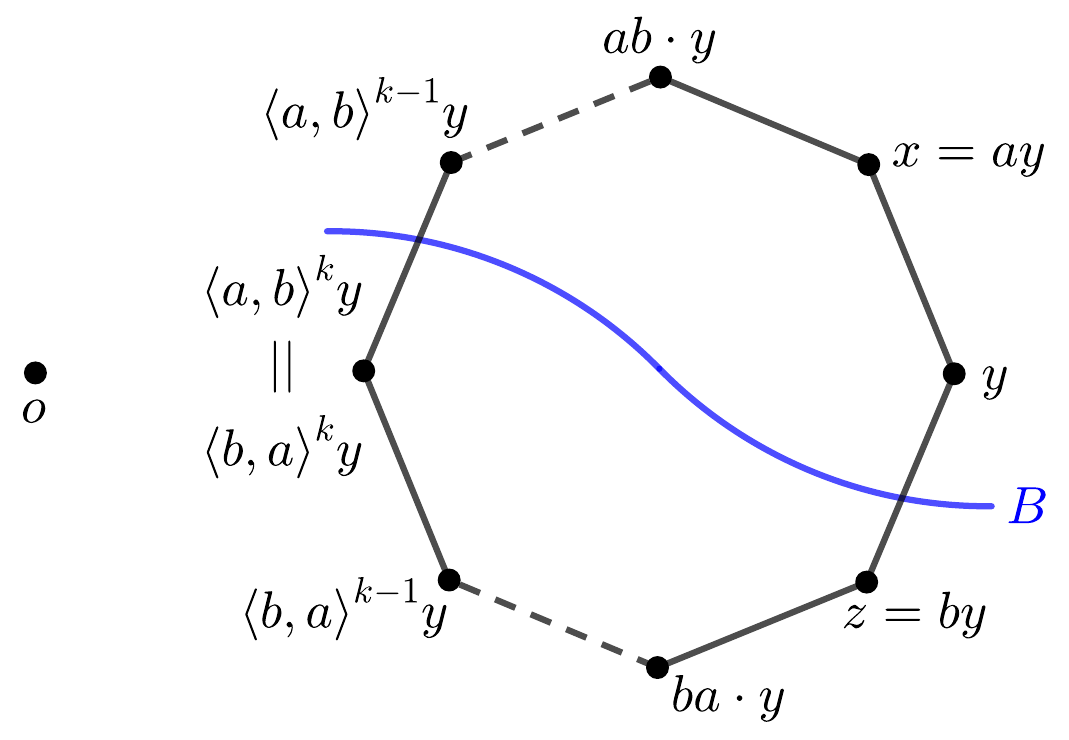}
\end{center}
\end{minipage}

\medskip \noindent
Observe that $k \geq 1$, which amounts to saying that $A \neq B$. Indeed, otherwise $x$ and $z$ would be adjacent, and the triangle condition provided by Proposition~\ref{prop:TriangleConditionRotation} combined with Lemma~\ref{lem:InterCliqueRotation} would imply that $d(o,y)=d(o,x)$. Therefore, our dihedral cycle is constructed, concluding the proof of the lemma.
\end{proof}

\begin{proof}[Proof of Theorem~\ref{thm:RotationPlurilith}.]
We already know from Proposition~\ref{prop:TriangleConditionRotation} that the triangle condition is satisfied and from Lemma~\ref{lem:InterCliqueRotation} that two $3$-cycles intersecting along at least an edge lie in a common clique. It remains to deal with convex even cycles.

\begin{claim}\label{claim:ConvexDihedral}
In $X$, an even cycle is convex if and only if it is dihedral. 
\end{claim}

\noindent
Let $C$ be a convex even cycle. Fix two opposite vertices $p,q \in C$ and let $a,b \in C$ denote the two neighbours of $p$. By applying Lemma~\ref{claim:CCZero} to these vertices, we know that the edges $[p,a],[p,b]$ span a dihedral cycle $C'$ such that the opposite vertex of $p$ lies in $I(p,q)$. Because $C$ is convex, necessarily $C=C'$. A fortiori, $C$ is dihedral. Conversely, we need to show that dihedral cycles are convex. By reproducing the proof of Proposition~\ref{prop:ConvexCriterion} word for word, but replacing convex even cycles with dihedral cycle - which is possible thanks to Proposition~\ref{prop:TriangleConditionRotation} and Lemma~\ref{lem:CycleConditionRotation} - we know that a subgraph $Y$ in $X$ is convex if and only if it is connected and locally convex in the sense that every dihedral cycle with at least half its length in $Y$ must actually lie entirely in $Y$. But a dihedral cycle is entirely determined by two consecutive edges, so it has to be convex, concluding the proof of Claim~\ref{claim:ConvexDihedral}.

\medskip \noindent
Finally, as already said, a dihedral cycle is uniquely determined by two consecutive edges, so the intersection between any two distinct dihedral cycles must contain at most one edge. The similar property for convex even cycles then follows from Claim~\ref{claim:ConvexDihedral}. 
\end{proof}

\subsection{Algebraic structure of rotation groups}

\noindent
In this section, our goal is to show that a rotation group is naturally isomorphic to a periagroup, as defined by Definition~\ref{def:Periagroups}. Formally:

\begin{prop}\label{prop:RotationPeriagroup}
Let $(G \curvearrowright X, \mathcal{R})$ be a rotation system. Fix a basepoint $o \in X$ and let $\mathscr{C}_o$ denote the set of cliques containing $o$. Then $G$ is a periagroup with $\mathcal{R}_o:= \{ \mathrm{stab}_\circlearrowright(C) \mid C \in \mathscr{C}_o\}$ as a basis. More precisely, let $\Gamma$ denote the graph whose vertex-set is $\mathscr{C}_o$ and whose edges connect two cliques if their hyperplanes are transverse; and let $\lambda : E(\Gamma) \to \mathbb{N}_{\geq 2}$ denote the map $(C_1,C_2) \mapsto  \pi / \measuredangle (C_1,C_2)$, where $\measuredangle(C_1,C_2)$ is the angle between the hyperplanes containing $C_1,C_2$. Then the identity map $\mathcal{R}_o \to \mathcal{R}_o$ induces an isomorphism $\Pi(\Gamma,\lambda,\mathcal{R}_o) \to G$.
\end{prop}

\begin{proof}
It follows from Proposition~\ref{prop:FreeTransitive} that $X$ coincides with $\mathrm{Cayl}(G, \mathcal{R}_o)$. Moreover, we know from Theorem~\ref{thm:RotationPlurilith} that $X$ is a mediangle graph and it follows from Lemma~\ref{lem:ShorteningPath} that $X$ can be made simply connected by filling in with discs its convex cycles. And the convex even cycles are described by Claim~\ref{claim:ConvexDihedral}, namely they are dihedral. Thus, we can compute an explicit presentation for $G$. In order to prove our proposition, it suffices to notice that this presentation coincides with $\Pi(\Gamma,\lambda,\mathcal{R}_o)$.

\medskip \noindent
Let $e,f \subset X$ be two edges containing $o$ that span a convex cycle $C$. Let $a \in G$ (resp. $b \in G$) denote the element of the rotative-stabiliser of the clique containing $e$ (resp. $f$) that sends $o$ to the second endpoint of $e$ (resp. $f$). If $C$ has length $3$, then $e,f$ belongs to the same clique $Q$ and $ba^{-1}=c$ is the relation in $\mathrm{stab}_\circlearrowright(Q) \in \mathcal{R}_o$ associated to $C$ where $c \in \mathrm{stab}_\circlearrowright(Q)$ is the element that sends the second endpoint of $e$ to the second endpoint of $f$. If $C$ has even length, then we know from Claim~\ref{claim:ConvexDihedral} that is dihedral, so its associated relation is $\langle a,b \rangle^k = \langle b,a \rangle^k$ where $k$ is half the length of $C$, which also coincides with $\pi/ \measuredangle$ where $\measuredangle$ denotes the angle between the hyperplanes given by the cliques containing $e,f$. 

\medskip \noindent
Thus, the presentation we get coincides with the relative presentation of $\Pi(\Gamma, \lambda, \mathcal{R}_o)$ given by Definition~\ref{def:Periagroups}. 
\end{proof}

\section{Periagroups}\label{section:Periagroups}

\subsection{Algebraic structure}

\noindent
In this section, we record a few results on the algebraic structure of periagroup that will be useful in the next section.

\medskip \noindent
Let $\Gamma$ be a graph, $\lambda : E(\Gamma) \to \mathbb{N}_{\geq 2}$ a labelling, and $\mathcal{G}= \{G_u \mid u \in V(\Gamma)\}$ a collection of groups indexed by $V(\Gamma)$. Assume that, for every edge $\{u,v\} \in E(\Gamma)$, if $\lambda(u,v)>2$ then $G_u,G_v$ are both cyclic of order two. Let $\Phi \leq \Gamma$ denote the subgraph given by the vertices labelled by vertex-groups that are cyclic of order two, and $\Psi \leq \Gamma$ the subgraph given by the vertices labelled by vertex-groups of order $>2$. Roughly speaking, $\Phi$ is the Coxeter part of our periagroup and $\Psi$ its graph product part.

\medskip \noindent
Our periagroup $\Pi := \Pi(\Gamma,\lambda,\mathcal{G})$ clearly projects onto the Coxeter group $C(\Phi)$. We begin by observation that the kernel of this morphism is naturally a graph product of groups. Before stating a precise result, we need to introduce some notation. 
\begin{itemize}
	\item For every vertex $u \in \Psi$, we denote by $\ell(u)$ the (standard) subgroup of $C(\Phi)$ generated by the vertices of $\Phi$ that are connected to $u$ by an edge. 
	\item Let $\Lambda$ be the graph whose vertex-set is $\{ g \ell(u) \mid g \in C(\Phi), u \in \Psi\}$ and whose edges connect two cosets $g \ell(u), h\ell(v)$ if they intersect in $C(\Phi)$ and if $u,v$ are adjacent in~$\Psi$.
	\item For all $u \in \Psi$ and $g \in C(\Phi)$, set $H_{g\ell(u)}:= G_u$ and $\mathcal{H}:= \{ H_{g \ell(u)} \mid g \in C(\Phi), u \in \Psi\}$.
\end{itemize}
Our claim above is justified by the following statement:

\begin{prop}\label{prop:SemiDirect}
The periagroup $\Pi$ is isomorphic to $\Lambda \mathcal{H} \rtimes C(\Phi)$ where $C(\Phi)$ acts on the graph product $\Lambda \mathcal{H}$ by permuting the vertex-groups according to its action on $\Lambda$ by left-multiplication. More precisely, $\Pi = P \rtimes C(\Phi)$ where $P$ is a graph product with $\{ gG_ug^{-1} \mid u \in \Psi, g \in C(\Phi) \text{ with } \mathrm{tail}(g) \cap \ell(u) = \emptyset \}$ as a basis.
\end{prop}

\noindent
In a Coxeter group, say $C(\Xi)$, the \emph{tail} of an element $g$ refers to $\mathrm{tail}(g):= \{ s \in \Xi \mid |gs|<|g| \}$, where $|\cdot|$ denotes the word length associated to the canonical generating set of $C(\Xi)$. Before proving Proposition~\ref{prop:SemiDirect}, we record the following observation:

\begin{lemma}\label{lem:Tail}
Let $\Xi$ be a labelled graph and $H \leq C(\Xi)$ a standard subgroup. For every $g \in C(\Xi)$, the coset $gH$ contains a unique element $m$ satisfying $\mathrm{tail}(m) \cap H = \emptyset$. It coincides with the unique element of $gH$ of minimal length in $C(\Xi)$.
\end{lemma}

\begin{proof}
First, we claim that $gH$ contains an element whose tail is disjoint from $H$. Fix an arbitrary element $m_0 \in gH$. If its tail is disjoint from $H$, we are done. Otherwise, there exists a generator $s \in H$ such that $|m_0s|<|m_0|$. As a consequence of the Exchange Condition (see \cite[page~35]{Davis}), we can write $m_0$ as $m_1s$ for some $m_1 \in C(\Xi)$ satisfying $|m_1|<|m_0|$. Observe that $m_1$ also belongs to $gH$. By iterating the argument, we construct a sequence of elements $m_0,m_1, \ldots$ in $gH$ of smaller and smaller length. The process has to stop, and we get an element $m$ whose tail is disjoint from $H$, as desired.

\medskip \noindent
Next, observe that, for every $k \in gH$, there exists some $h \in H$ such that $k=mh$. Because $\mathrm{tail}(m) \cap H = \emptyset$, we have $|k|=|mh|=|m|+|h|$, which greater than $|m|$ unless $k=m$. Thus, $m$ is the unique element in $gH$ of minimal length.
\end{proof}

\begin{proof}[Proof of Proposition~\ref{prop:SemiDirect}.]
The semidirect product $\Lambda \mathcal{H} \rtimes C(\Phi)$ admits as a relative presentation
$$\left\langle u \in \Phi, H_{g\ell(v)}  \left( \begin{array}{l} v \in \Psi, \\ g \in C(\Phi) \end{array} \right) \left| \begin{array}{l} [H_{g \ell(u)}, H_{h \ell(v)}]=1 \text{ if $g \ell(u),h \ell(v)$ adjacent in $\Lambda$} \\ hH_{g \ell(u)} h^{-1} = H_{hg \ell(u)}, \ g,h \in C(\Phi), u \in \Psi \end{array}, \mathrm{Cox} \right. \right\rangle,$$
where $\mathrm{Cox}$ is the set of relations coming from the Coxeter presentation of $C(\Phi)$, namely
$$u^2=1, \ u \in \Phi \text{ and } \langle u,v \rangle^{\lambda(u,v)}=\langle v,u \rangle^{\lambda(u,v)}, \ \{u,v\} \in E(\Phi).$$
In order to simplify the presentation, we assume that $\mathrm{tail}(g) \cap \ell(u) = \emptyset$ in our $H_{g \ell(u)}$ and we use the relation $H_{g \ell(u)}= g H_{\ell(u)} g^{-1}$ to remove redundant generators. 

\medskip \noindent
Observe that our commutation relations reduce to: $[H_{\ell(u)},H_{\ell(v)}]=1$ for all $u,v \in \Psi$ adjacent. Indeed, if $g \ell(u)$ and $h \ell(v)$ are adjacent in $\Lambda$, then the cosets $g \ell(u), h \ell(v)$ intersect, which implies that there exist $a \in \ell(u)$ and $b \in \ell(v)$ such that $h=gab$. Then
$$\begin{array}{lcl} \left[H_{g \ell(u)},H_{h\ell(v)}\right] & = & \left[gH_{\ell(u)}g^{-1}, h H_{\ell(v)}h^{-1}\right]= g \left[ H_{\ell(u)}, a H_{\ell(v)} a^{-1}\right] g^{-1} \\ \\ & = & ga \left[H_{\ell(u)}, H_{\ell(v)}\right] a^{-1}g^{-1}, \end{array}$$
which implies that our reduction holds. 

\medskip \noindent
Our second set of relations, describing the action by conjugation of $C(\Phi)$ on $\Lambda \mathcal{H}$, reduces to: $gH_{\ell(u)}g^{-1}= H_{ \ell(u)}$, $g \in \ell(u), u \in \Psi$. Or, equivalently, to: $[H_{\ell(u)},v]=1$, $u \in \Psi$, $v \in \Phi$ such that $\{u,v\} \in E(\Gamma)$ and $\lambda(u,v)=2$. 

\medskip \noindent
Thus, we have simplified our presentation as
$$\left\langle u \in \Phi, H_{\ell(v)} \ (v \in \Psi) \left| \begin{array}{l} \left[ H_{\ell(u)},H_{\ell(v)} \right]=1, \ u,v \in \Psi \text{ adjacent} \\ \left[ H_{\ell(u)},v \right]=1, \ u \in \Psi, v \in \Phi \text{ adjacent with } \lambda(u,v)=2 \end{array}, \mathrm{Cox} \right. \right\rangle,$$
which coincides with the presentation of $\Pi$ given by Definition~\ref{def:Periagroups}. This concludes the proof of our proposition.
\end{proof}

\begin{cor}\label{cor:InterVertexGroups}
For all $u,v \in \Gamma$ and $g \in \Pi$, if $gG_ug^{-1}$ and $G_v$ intersect non-trivially, then they are equal. Moreover, if $u \neq v$, then $G_u,G_v$ are both cyclic of order two.
\end{cor} 

\begin{proof}
If $u \in \Phi$ and $v \in \Psi$, or if $v \in \Phi$ and $u \in \Psi$, then we know from Proposition~\ref{prop:SemiDirect} that $gG_ug^{-1}$ and $G_v$ interest trivially. If $u,v \in \Phi$, it is clear that $gG_ug^{-1}$ and $G_v$ are equal if they intersect non-trivially (since they are two groups of order two). If $u,v \in \Psi$, it follows from Proposition~\ref{prop:SemiDirect} that $\Pi$ contains a subgroup isomorphic to a graph product having $G_v$ and $gG_ug^{-1}$ as vertex-groups. But, in a graph product, two distinct vertex-groups always have a trivial intersection. They are not conjugate either. Since $C(\Phi)$ permutes the vertex-groups of $\Lambda \mathcal{H}$ without sending a conjugate of $G_u$ to a conjugate of $G_v$ if $u \neq v$, our corollary is proved.
\end{proof}

\begin{remark}
Similarly, the periagroup $\Pi$ surjects onto the graph product $\Psi \mathcal{G}$, and this quotient map splits, giving a decomposition of $\Pi$ as a semidirect product $C(\Lambda) \rtimes \Psi \mathcal{G}$. Here, $\Lambda$ is the graph whose vertex-set is $\{g \langle \mathrm{link}_\Psi(u) \rangle \mid g \in \Psi \mathcal{G},u \in \Phi\}$ and whose edges connect two cosets $g \langle \mathrm{link}_\Psi(u) \rangle$, $h \langle \mathrm{link}_\Psi(v) \rangle$ if they intersect and if $u,v$ are adjacent in $\Phi$, such an edge of $\Lambda$ being labelled by $\lambda(u,v)$. The proof is a straightforward adaptation of the proof of Proposition~\ref{prop:SemiDirect}. As this decomposition will not be used in the sequel, we omit the details. 
\end{remark}

\noindent
By adding commutation relations from every vertex-group from the graph product part to all the other vertex-groups, we obtain another interesting quotient, namely $\Pi \twoheadrightarrow C(\Phi) \oplus \bigoplus_{u \in \Psi} G_u$. As a consequence, one immediately deduces that subgroups generated by different vertex-groups intersect trivially. More generally:

\begin{lemma}\label{lem:DisjointSubgraphsSubgroups}
For any two disjoint subgraphs $\Lambda_1,\Lambda_2 \leq \Gamma$, the subgroups $\langle \Lambda_1 \rangle$ and $ \langle \Lambda_2 \rangle$ intersect trivially.\qed
\end{lemma}

\noindent
Here, we use the following convenient notation: given a subgraph $\Lambda \leq \Gamma$, we denote by $\langle \Lambda \rangle$ the subgroup of $\Pi$ generated by the vertex-groups indexing the vertices of $\Lambda$.

\subsection{Periagroups as rotation groups}

\noindent
In this section, we prove the last piece of Theorem~\ref{thm:BigIntro} by showing that every periagroup arises as a rotation group.

\begin{prop}\label{prop:PeriagroupsAsRotation}
Let $\Gamma$ be a graph, $\lambda : E(\Gamma) \to \mathbb{N}_{\geq 2}$ a labelling map, and $\mathcal{G}= \{ G_u \mid u \in V(\Gamma)\}$ a collection of groups indexed by the vertices of $\Gamma$. Assume that, for every edge $\{u,v\} \in E(\Gamma)$, if $\lambda(u,v)>2$ then $G_u,G_v$ are both cyclic of order two. Letting $\Pi$ denote the periagroup $\Pi(\Gamma, \lambda, \mathcal{G})$, $(\Pi \curvearrowright \mathrm{Cayl}(\Pi,\mathcal{G}), \{ gGg^{-1} \mid g \in \Pi, G \in \mathcal{G} \})$ is a rotation system.
\end{prop}

\noindent
From now on, we follow the notations used in Proposition~\ref{prop:PeriagroupsAsRotation}. 

\begin{lemma}\label{lem:CliquePeria}
The cliques in $X:=\mathrm{Cayl}(\Pi, \mathcal{G})$ coincide with the cosets of the vertex-groups $gG_u$, $g \in \Pi$, $u \in \Gamma$. 
\end{lemma}

\begin{proof}
Let $C$ be a $3$-cycle in the Cayley graph. If $g \in C$ is one of its vertices, its two other vertices can be written as $ga, gb$ where $a \in G_u$, $b \in G_v$ for some $u,v \in \Gamma$. Moreover, because $ga$ and $gab$ are adjacent, there must exist some $w \in \Gamma$ such that $a^{-1}b \in G_w$. Thus, $G_w \cap \langle G_u,G_v \rangle$ contains a non-trivial element, namely $a^{-1}b$. Lemma~\ref{lem:DisjointSubgraphsSubgroups} implies that $w \in \{u,v\}$. Similarly, $G_u \cap \langle G_v,G_w \rangle$ contains $a = b^{-1} \cdot (a^{-1}b)$ and $G_v \cap \langle G_u ,G_w \rangle$ contains $b=a \cdot a^{-1}b$. So we must have $u=v=w$. Thus, we have proved that every $3$-cycle in $X$ has all its edges labelled by the same vertex-group. This implies that every complete subgraph lies in the coset of some vertex-group. Since cosets of vertex-groups are complete subgraphs, the desired conclusion follows. 
\end{proof}

\begin{proof}[Proof of Proposition~\ref{prop:PeriagroupsAsRotation}.]
Set $\mathcal{R}:=\{ gGg^{-1} \mid g \in \Pi, G \in \mathcal{G} \}$ and $X:=\mathrm{Cayl}(\Pi, \mathcal{G})$. First, we claim that $(\Pi \curvearrowright X,\mathcal{R})$ is a rotation presystem. It is clear that $\mathcal{R}$ generates $\Pi$ and that it is stable under conjugation. It follows from Lemma~\ref{lem:CliquePeria} that every clique of $X$ has a (unique) rotative-stabiliser, namely its own stabiliser. And, conversely, the same lemma shows every subgroup in $\mathcal{R}$ is the (rotative-)stabiliser of a clique. Thus, our claim is proved.

\medskip \noindent
Following \cite[Lemma~3.3.5]{Davis}, we construct a representation of $\Pi$ that will correspond, a fortiori, to the action of $\Pi$ of the set of sectors of the mediangle graph $X$. Of course, since we do not know yet that $X$ is a mediangle graph, we have to define the representation directly. Let $\mathscr{S}$ the set of the pairs $(G,g)$ where $G \in \mathcal{R}$ and $g \in G$. For every generator $s \in \Pi$, define
$$\varphi_s : \left\{ \begin{array}{ccc} \mathscr{S} & \to & \mathscr{S} \\ (G,g) & \mapsto & (sGs^{-1}, s^{\delta(G,s)}g) \end{array} \right.,$$
where $\delta(G,s):= 1$ if $s \in G$ and $0$ otherwise, and where the notation $a^b$ refers to $b^{-1}ab$. A direct computation shows that $\varphi_s \circ \varphi_{s^{-1}} = \mathrm{Id} = \varphi_{s^{-1}} \circ \varphi_s$, so $\varphi_s$ defines a bijection of $\mathscr{S}$. For every word $w:=s_1 \cdots s_n$ of generators, we set $\varphi_w:= \varphi_{s_1} \circ \cdots \circ \varphi_{s_n}$. Our goal is to show that $w \mapsto \varphi_w$ induces a morphism $\Pi \to \mathrm{Bij}(\mathscr{S})$. It suffices to show that, if $w$ is a relation from the presentation of $\Pi$, then $\varphi_w= \mathrm{Id}$. First of all, notice that a straightforward proof by induction shows that
$$\varphi_{s_1 \cdots s_n} : (G,g) \mapsto \left( (s_1 \cdots s_n) \cdot G \cdot (s_1 \cdots s_n)^{-1}, \prod\limits_{i=1}^n s_i^{\delta \left( G,s_i^{s_{i+1}\cdots s_n} \right) } \cdot g \right).$$
We distinguish three types of relations. First, assume that $w=1$ is a relation in some vertex-group $G_u$. Observe that, as a consequence of Corollary~\ref{cor:InterVertexGroups}, for every $1 \leq i \leq n$ and for every $G \in \mathcal{R}$, $s_i^{s_{i+1}\cdots s_n}$ belongs to $G$ if and only if $G=G_u$. Consequently, $\varphi_{w}(G,g)=(G,g)$ for all $G \neq G_u$ and $g \in G$; and $\varphi_{w}(G_u,g)=(G_u, s_1 \cdots s_n \cdot g)=(G_u,g)$ for every $g \in G_u$. Thus, $\varphi_{w}= \mathrm{Id}$ as desired.

\medskip \noindent
Next, assume that $w=1$ is a commutation relation between two vertex-groups, i.e. $w=aba^{-1}b^{-1}$ where $a \in G_u$, $b \in G_v$ for some adjacent vertices $u,v \in \Gamma$. A direct computation shows that, for all $G \in \mathcal{R}$ and $g \in G$, we have
$$\varphi_w(G,g)= \left( G, a^{\delta(G,a)} b^{\delta(G,b)} a^{-\delta(G,a)} b^{-\delta(G,b)} \cdot g \right) = (G, g)$$
since $a$ and $b$ commute in $\Pi$. So $\varphi_w= \mathrm{Id}$ as desired.

\medskip \noindent
Finally, assume that $w=1$ is a dihedral relation between two vertex-groups of order two, say $(uv)^{\lambda(u,v)}=1$ where we identify for convenience the vertices $u,v \in \Gamma$ with the generators of their corresponding vertex-groups. For all $G \in \mathcal{R}$ and $g \in G$, we have
$$\varphi_w(G,g)=\left( G, \prod_{i=1}^{\lambda(u,v)} u^{\delta \left(G,u^{v(uv)^{\lambda(u,v)-i}} \right)} v^{\delta \left( G,v^{(uv)^{\lambda(u,v)-i}} \right)} \cdot g \right).$$
If all the $\delta(\cdot,\cdot)$ is zero, then we have $\varphi_w(G,g)=(G,g)$ as desired. Otherwise, it follows from Corollary~\ref{cor:InterVertexGroups} that $G$ is a conjugate of $G_u$ or $G_v$ in $\langle G_u,G_v \rangle$. But we know from Proposition~\ref{prop:SemiDirect} that $\langle G_u,G_v \rangle$ is isomorphic to a dihedral group of order $2\lambda(u,v)$ and the restrictions of $\varphi_u,\varphi_v$ to $\{ (gG_zg^{-1},\epsilon) \mid \epsilon \in \{1,z\}, g \in \langle G_u,G_v \rangle, z \in \{u,v\} \}$ coincides with the actions of $u$ and $v$ on the halfspaces of the Cayley graph of the dihedral group $\langle u,v \rangle$ (this how we constructed the $\varphi_s$). So $\varphi_w(G,g)=(G,g)$ in all these remaining cases, proving that $\varphi_w= \mathrm{Id}$.

\medskip \noindent
Thus, we have a well-defined morphism $\Pi \to \mathrm{Bij}(\mathscr{S})$. We are now ready to conclude the proof of our proposition. So let $C$ be a clique in $X$ with two vertices $x,y \in C$ and let $J$ denote the union of all the cliques having the same rotative-stabiliser as $C$. We want to prove that $J$ separates $x$ and $y$. It follows from Lemma~\ref{lem:CliquePeria} that, up to translating by an element of $\Pi$, we can assume without loss of generality that $C=G_u$ for some $u \in V(\Gamma)$, that $x=1$, and that $y=s$ for some generator $s$. Given a path $\gamma$ in $X$ from $1$ to $s$, the word $w=s_1 \cdots s_n$ labelling it has to be equal to $s$ in $\Pi$. Consequently,
$$(G_u,s)= \varphi_s(G_u,1)=\varphi_w(G_u,1)= \left( G_u, \prod\limits_{i=1}^n s_i^{\delta\left( G_u, s_i^{s_{i+1}\cdots s_n} \right)} \right).$$
Necessarily, there exists some $1 \leq k \leq n$ such that $\delta\left( G_u,s_k^{s_{k+1}\cdots s_n} \right) \neq 0$, which amounts to saying that $s_k^{s_{k+1} \cdots s_n} \in G_u$. After conjugation by $s$, the latter becomes $(s_1 \cdots s_k) \cdot s_{k+1} \cdot (s_1 \cdots s_k)^{-1} \in G_u$. But this element also belongs to the rotative-stabiliser of the clique $Q$ that contains the edge of $\gamma$ between $s_1 \cdots s_k$ and $s_1 \cdots s_{k+1}$. Since $Q$ and $C$ have the same rotative-stabiliser as a consequence of Corollary~\ref{cor:InterVertexGroups} (which applies thanks to Lemma~\ref{lem:CliquePeria}), we conclude that $\gamma$ crosses $J$. Thus, our proposition is finally proved.
\end{proof}

\begin{proof}[Proof of Theorem~\ref{thm:BigIntro}.]
The theorem is the combination of Proposition~\ref{prop:RotationPeriagroup}, Theorem~\ref{thm:RotationPlurilith}, Proposition~\ref{prop:FreeTransitive}, and Proposition~\ref{prop:PeriagroupsAsRotation}. 
\end{proof}

\subsection{Normal form}

\noindent
Fix a graph $\Gamma$, a collection of groups $\mathcal{G}=\{G_u \mid u \in V(\Gamma)\}$ indexed by the vertices of $\Gamma$, and a labelling $\lambda : E(\Gamma) \to \mathbb{N}_{\geq 2}$ such that, for every edge $\{u,v\} \in E(\Gamma)$, if $\lambda(u,v) >2$ then $G_u,G_v$ are both cyclic of order two. We denote the associated periagroup $\Pi(\Gamma,\lambda,\mathcal{G})$ by $\Pi$.

\medskip \noindent
A \emph{word} in $\Pi$ is a product $g_1 \cdots g_n$ where $n \geq 0$ and where, for every $1 \leq i \leq n$, $g_i$ belongs to $G_i$ for some $G_i \in \mathcal{G}$; the $g_i$'s are the \emph{syllables} of the word, and $n$ is the \emph{length} of the word. Clearly, the following operations on a word does not modify the element of $\Pi$ it represents:
\begin{description}
	\item[(reduction)] remove the syllable $g_i$ if $g_i=1$;
	\item[(fusion)] if $g_i,g_{i+1} \in G$ for some $G \in \mathcal{G}$, replace the two syllables $g_i$ and $g_{i+1}$ by the single syllable $g_ig_{i+1} \in G$;
	\item[(dihedral relation)] if there exist $\{u,v\} \in E(\Gamma)$ such that $g_i\cdots g_{i+\lambda(u,v)-1}= \langle a,b \rangle^{\lambda(u,v)}$ for some $a \in G_u$, $b \in G_v$, then replace this subword with $\langle b,a \rangle^{\lambda(u,v)}$.
\end{description}
A word is \emph{graphically reduced} if its length cannot be shortened by applying these elementary moves.

\medskip \noindent
The following statement extends what is known about the word problem in Coxeter groups \cite{MR0254129} (see also \cite[Theorem~3.4.2]{Davis}) and graph products of groups \cite{GreenGP}.

\begin{prop}
Let $g \in \Pi$. A word $w$ representing $g$ has minimal length among all the words representing $g$ if and only if it is graphically reduced. Moreover, any two graphically reduced words representing $g$ can be obtained from each other by applying dihedral relations.
\end{prop}

\begin{proof}
There is a bijection between the words representing $g$ and the path in $X:= \mathrm{Cayl}(\Pi,\mathcal{G})$ from $1$ to $g$. (Here, we allow paths to stay at a given point in order to take into account the possibly trivial syllables of words.) But we know from Theorem~\ref{thm:BigIntro} that $X$ is a mediangle graph, and, as a consequence of Lemma~\ref{lem:ShorteningPath}, a path between two given vertices has minimal length if and only if it cannot be shortened by removing backtracks, shortening triangles, and applying flips. But removing a backtrack to the path amounts to applying a fusion and a reduction to the corresponding word; shortening a triangle amounts to applying a fusion; and applying a flip amounts to applying a dihedral relation. Consequently, a word has minimal length if and only if it is graphically reduced. For the same reason, the second assertion of our proposition follows from Lemma~\ref{lem:FlipGeod}. 
\end{proof}

\begin{cor}
A periagroup of groups with solvable word problem has a solvable word problem.
\end{cor}

\begin{proof}
Because there exist only finitely many words of a given length, one can check whether or not a given word is graphically reduced with the word problem is already solvable in vertex-groups (in order to determine whether a reduction can be applied). If our word is not graphically reduced, we shorten its length by applying a reduction or a fusion, and we iterate. After a finite amount of time, either we obtain the empty word and our initial word represents the trivial element; or we obtain a non-empty word and our initial word represents a non-trivial element of $\Pi$. 
\end{proof}

\noindent
We conclude this section by observing that one can easily deduce from our geometric framework an analogue of the Exchange Condition that holds for Coxeter groups (see for instance \cite[page~35]{Davis}). 

\begin{prop}
Let $g \in G$ be an element and $s \in G_u$ a generator. If $|gs|=|g|$, then $g$ can be represented as a graphically reduced word of the form $ws'$ where $s' \in G_u \backslash \{1,s^{-1}\}$; if $|gs|<|g|$, then $g$ can be represented as a graphically reduced word of the form $w s^{-1}$. 
\end{prop}

\begin{proof}
Assume first that $|gs|=|g|$. Since $|\cdot|$ coincides with the distance from $1$ in $X:= \mathrm{Cayl}(\Pi,\mathcal{G})$, we can apply the triangle condition and find a common neighbour $h$ of $g$ and $gs$ such that $|h|=|g|-1$. Let $s'$ be a generator such that $g=hs'$. Observe that $s' \neq 1,s^{-1}$ since $h,g,gs$ are pairwise distinct. Moreover, we deduce from Lemma~\ref{lem:CliquePeria} that $s'$ belongs to the same vertex-groups as $s$. We conclude that, if $w$ is graphically reduced word representing $h$, then $ws'$ is a graphically reduced word representing $g$, as desired.

\medskip \noindent
Next, assume that $|gs|=|g|-1$. The graphically reduced word labelling a geodesic in $X$ from $1$ to $g$ passing through $gs$ can be written as $wa$ where $a$ is the generator labelling the last edge $[gs,g]$. Necessarily, we have $a=s^{-1}$, as desired. 
\end{proof}

\section{Applications}

\subsection{Rotation subgroups}

\noindent
This subsection is dedicated to Theorem~\ref{thm:IntroRotationSub}. Its proof will be a rather easy consequence of the following geometric observation:

\begin{thm}\label{thm:PingPong}
Let $G$ be a group acting on a mediangle graph $X$ and let $\mathcal{J}$ be a $G$-invariant collection of hyperplanes. Assume that, for every $J \in \mathcal{J}$, $\mathrm{stab}_\circlearrowright(J)$ permutes freely and transitively the sectors delimited by $J$. There exists a convex subgraph $Y \subset X$ such that $G$ decomposes as the semidirect product $\mathrm{Rot} \rtimes \mathrm{stab}(Y)$, where $\mathrm{Rot}:= \langle \mathrm{stab}_\circlearrowright (J), J \in \mathcal{J} \rangle$. Moreover, $\mathrm{Rot}$ is a periagroup with $\{ \mathrm{stab}_\circlearrowright(J), \text{ $J$ tangent to $Y$}\}$ as a basis.
\end{thm}

\noindent
Here, we say that a hyperplane $J$ is \emph{tangent} to a subgraph $Y$ if there exists an edge in $J$ that has exactly one endpoint in $Y$. 

\begin{proof}[Proof of Theorem~\ref{thm:PingPong}.]
For any two vertices $x,y \in X$, we denote by $\beta(x,y)$ the number of hyperplanes in $\mathcal{J}$ that separate $x$ and $y$. Because $\delta$ defines a pseudo-metric, the relation $\sim$ given by $\beta(\cdot,\cdot)=0$ defines an equivalence relation on the vertices of $X$. Let $Y$ denote a $\sim$-class, thought of as a subgraph of $X$. Alternatively, $Y$ is the intersection of all the sectors delimited by hyperplanes in $\mathcal{J}$ that contain a given basepoint. So $Y$ is a convex subgraph. Let $\Omega$ denote the graph whose vertex-set is $X / \sim$ and whose edges connect two classes at $\delta$-distance $1$. The picture to keep in mind is that $\Omega$ is the graph dual to the decomposition of $X$ induces by $\mathcal{J}$. Because $\mathcal{J}$ is $G$-invariant, $G$ naturally acts on $\Omega$. Setting $\mathcal{R}$ as the collection of rotative-stabilisers of hyperplanes in $\mathcal{J}$, we want to prove that:

\begin{claim}\label{claim:MergeRotation}
$(\mathrm{Rot} \curvearrowright \Omega, \mathcal{R})$ is a rotation system.
\end{claim}

\noindent
First, it is clear that $\mathcal{R}$ is stable under conjugation and that it generates $\mathrm{Rot}$. 

\medskip \noindent
Next, let $C$ be a clique of $\Omega$. There must exist a hyperplane $J \in \mathcal{J}$ such that representatives of $C$ lie in pairwise distinct sectors delimited by $J$. Consequently, if $K$ is a complete subgraph lying in $J$ and containing exactly one vertex in each of these sectors, then the quotient modulo $\sim$ induces a bijection $K \to C$. Because $\mathrm{stab}_\circlearrowright(J)$ stabilises $K$ and permutes freely and transitively the sectors delimited by $J$, it follows that $\mathrm{stab}_\circlearrowright(J)$ acts freely and transitively on the vertices of $C$. For any other hyperplane $J' \in \mathcal{J}$ and any element $g \in \mathrm{stab}_\circlearrowright(J')$, $K$ and $g \cdot K$ lie in distinct sectors delimited by $J'$, so $g$ cannot stabilise $C$. Thus, $\mathrm{stab}_\circlearrowright(J)$ is the rotative-stabiliser of $C$. Moreover, every path in $\Omega$ between two vertices of $C$ lifts to a path in $X$ between the corresponding two vertices of $K$. Because these two vertices lie in distinct sectors delimited by $J$, the latter path must contain an edge in some clique $Q \subset J$ that connects two distinct sectors. Because $\mathrm{stab}_\circlearrowright(J)$ is clearly the rotative-stabiliser of $Q/ \sim$ as well, we conclude that removing the edges of all the cliques having $\mathrm{stab}_\circlearrowright(J)$ as rotative-stabiliser separates any two vertices in $C$.

\medskip \noindent
Finally, observe that, for every $J \in \mathcal{J}$ and for every clique $K \subset J$, $\mathrm{stab}_\circlearrowright(J)$ is the rotative-stabiliser of the clique $C:= K/\sim$. This concludes the proof of Claim~\ref{claim:MergeRotation}. 

\medskip \noindent
We deduce from Proposition~\ref{prop:RotationPeriagroup} that $\mathrm{Rot}$ is a periagroup with a basis given by the rotative-stabilisers of the cliques of $\Omega$ containing $Y$ (thought of as a vertex of $\Omega$). By construction, these stabilises coincide with the rotative-stabilisers of the hyperplanes in $\mathcal{J}$ tangent to $Y$ (thought of as a subgraph of $X$). This proves the second assertion of our theorem.

\medskip \noindent
We also deduce from Proposition~\ref{prop:FreeTransitive} that $\mathrm{Rot} \cap \mathrm{stab}(Y) = \{1\}$, so, since $\mathrm{Rot}$ is clearly normal in $G$, we have $\langle \mathrm{Rot}, \mathrm{stab}(Y) \rangle= \mathrm{Rot} \rtimes \mathrm{stab}(Y)$. In order to prove our theorem, it remains to show that $\mathrm{Rot}$ and $\mathrm{stab}(Y)$ generate $G$. Fix an element $g \in G$. Because $\mathrm{Rot}$ acts transitively on the vertices of $\Omega$, according to Proposition~\ref{prop:FreeTransitive}, there exists $r \in \mathrm{Rot}$ such that $rg \cdot Y =Y$. Hence $rg \in \mathrm{stab}(Y)$, and a fortiori $g \in \mathrm{Rot} \cdot \mathrm{stab}(Y)$. Our theorem is proved. 
\end{proof}

\noindent
Fix a graph $\Gamma$, a collection of groups $\mathcal{G}=\{G_u \mid u \in V(\Gamma)\}$ indexed by the vertices of $\Gamma$, and a labelling $\lambda : E(\Gamma) \to \mathbb{N}_{\geq 2}$ such that, for every edge $\{u,v\} \in E(\Gamma)$, if $\lambda(u,v) >2$ then $G_u,G_v$ are both cyclic of order two. We denote by $\Pi$ the corresponding periagroup $\Pi(\Gamma, \lambda, \mathcal{G})$.

\begin{lemma}\label{lem:RotativeStabPeria}
Any two cliques in $X:=\mathrm{Cayl}(\Pi, \mathcal{G})$ that lie in the same hyperplane have the same stabiliser.
\end{lemma}

\begin{proof}
Let $C_1,C_2 \subset X$ be two cliques lying in the same hyperplane. It suffices to consider the case where $C_1,C_2$ contains two opposite edges from some convex even cycle. First, assume that $C_1$ and $C_2$ are two opposite edges in a convex cycle of even length $>4$. As a consequence of Claim~\ref{claim:ConvexDihedral}, up to translating $C_1,C_2$ by an element of $\Pi$, there exist two adjacent vertex $u,v \in \Gamma$ satisfying $\lambda(u,v) \geq 3$ and two elements $a \in G_u$, $b \in G_v$ such that 
$$C_1=[1,a] \text{ and } C_2= \left\{ \begin{array}{ll} \langle a,b \rangle^{\lambda(u,v)} [1,a] & \text{if $\lambda(u,v)$ is even} \\ \langle a,b \rangle^{\lambda(u,v)} [1,b] & \text{if $\lambda(u,v)$ is odd} \end{array} \right..$$
Then $\langle a \rangle$ is the common stabiliser of $C_1$ and $C_2$. Next, assume that $C_1$ and $C_2$ contain two opposite edges from some $4$-cycle. It follows from Claim~\ref{claim:ConvexDihedral} and Lemma~\ref{lem:CliquePeria} that, up to translating by an element of $\Pi$, there exist two adjacent vertices $u,v \in \Gamma$ satisfying $\lambda(u,v)$ and an element $g \in G_v$ such that $C_1=G_u$ and $C_2=gG_u$. Because $g$ commutes with $G_u$, we conclude that $G_u$ is the common stabiliser of $C_1$ and $C_2$. 
\end{proof}

\begin{cor}\label{cor:RotativeStabPeria}
For every hyperplane $J$ of $X$, $\mathrm{stab}_\circlearrowright(J)$ is conjugate to a vertex-group and it permutes freely and transitively the sectors delimited by $J$
\end{cor}

\begin{proof}
It follows from Lemma~\ref{lem:RotativeStabPeria} that $\mathrm{stab}_\circlearrowright(J)$ coincides with the stabiliser of any clique in $J$. It follows from Lemma~\ref{lem:CliquePeria} that $\mathrm{stab}_\circlearrowright(J)$ is conjugate to a vertex-group $G$ and that its action on the set of sectors delimited by $J$ coincides with the action of $G$ on itself by left-multiplication. This proves our corollary.
\end{proof}

\begin{proof}[Proof of Theorem~\ref{thm:IntroRotationSub}.]
Let $\{ G_i \mid i \in I\}$ be a collection of conjugates of vertex-groups, and let $\mathrm{Rot}$ denote the subgroup it generates. We deduce from Lemma~\ref{lem:RotativeStabPeria} that, for every $i \in I$, there exists some hyperplane $J_i$ such that $G_i= \mathrm{stab}_\circlearrowright(J_i)$. Set $\mathcal{J}:= \mathrm{Rot} \cdot \{ J_i \mid i \in I\}$. It follows from Theorem~\ref{thm:PingPong} (which applies according to Corollary~\ref{cor:RotativeStabPeria}) that $\mathrm{Rot}$ is a periagroup with vertex-groups that are vertex-groups of $\Pi$. 
\end{proof}

\subsection{Parabolic subgroups}

\noindent
This subsection is dedicated to the proof of Theorem~\ref{thm:IntroParabolic}. Before turning to its proof, we focus on a local characterisation of gated subgraphs in mediangle graphs in the same spirit as Proposition~\ref{prop:ConvexCriterion}. 

\begin{prop}\label{prop:LocallyGated}
Let $X$ be a mediangle graph. A subgraph $Y \subset X$ is gated if and only if it is connected and locally gated. 
\end{prop}

\noindent
A subgraph $Y$ is \emph{locally gated} if every $3$-cycle having an edge in $Y$ lies entirely in $Y$ and if every convex even cycle having two consecutive edges in $Y$ lies entirely in $Y$. 

\begin{proof}[Proof of Proposition~\ref{prop:LocallyGated}.]
It is clear that a gated subgraph is connected, locally convex, and locally gated. From now on, we assume that $Y$ is connected, locally convex, and locally gated. We already know from Proposition~\ref{prop:ConvexCriterion} that $Y$ is convex. Fix a vertex $x \in X$ and let $x' \in Y$ be a vertex at minimal distance from $x$. We want to prove that, for every $y \in Y$, we have $x' \in I(x,y)$. For contradiction, assume that there exists some $y \in Y$ such that $x' \notin I(x,y)$; we choose such a $y$ with $d(y,x')$ minimal. Along a geodesic from $x'$ to $y$, let $a$ denote the neighbour of $y$ and $b$ the neighbour of $a$ distinct from $y$. Because $x' \notin I(x,y)$ but $x' \in I(x,a)$, we must have $d(x,y) \in \{ d(x,a), d(x,a)-1\}$. 

\begin{center}
\includegraphics[width=0.9\linewidth]{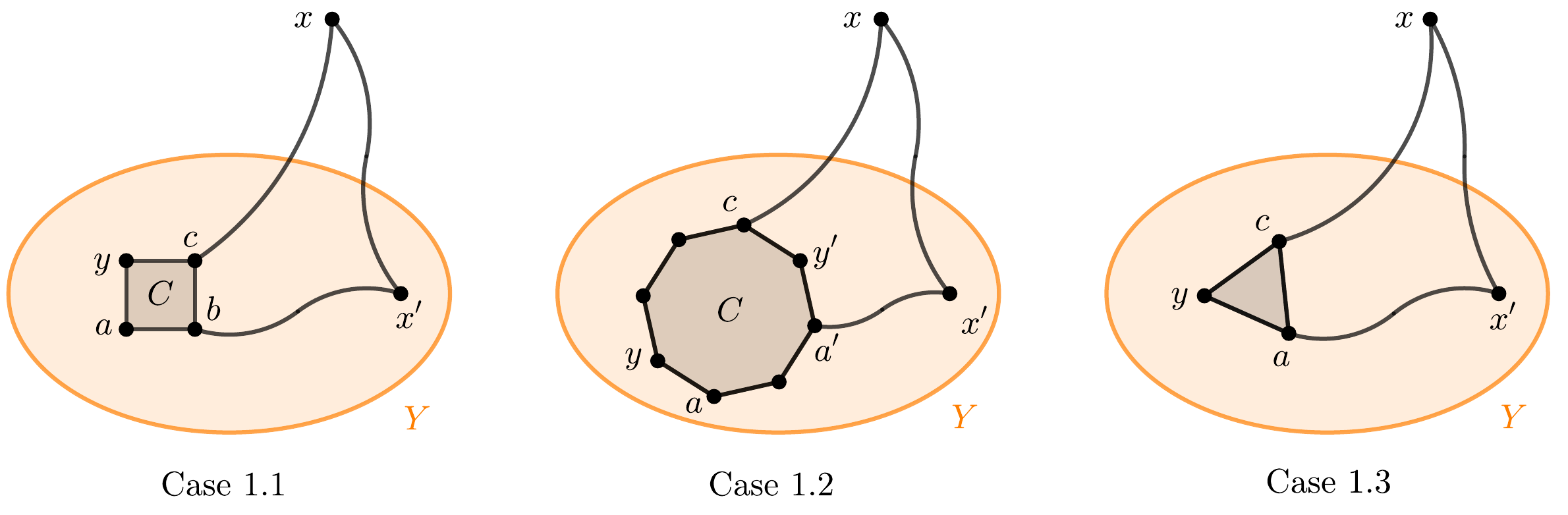}
\end{center}

\medskip \noindent
\underline{Case 1:} Assume first that $d(x,y)=d(x,a)-1$. Then the cycle condition applies and we find a convex even cycle $C$ spanned by the edges $[a,y],[a,b]$ such that the vertex $c$ opposite to $a$ belongs to $I(x,a)$. Observe that, because $Y$ is locally gated, $C \subset Y$. 

\medskip \noindent
\underline{Case 1.1:} Assume that $C$ has length four. Observe that 
$$d(x',c) \geq d(x',y)-d(y,c) = d(x',b)+2-d(y,c)=d(x',b)+1.$$
On the other hand, $d(x',c) \leq d(x',b)+d(b,c)= d(x',b)+1$. So $d(x',c)=d(x',b)+1$. In particular, $d(x',c)<d(x,y)$. Moreover, 
$$d(x,c)=d(x,a)-2=d(x,x')+d(x',b)-1= d(x,x')+d(x',c)-2,$$
so $x' \notin I(x,c)$. This contradicts our choice of $y$.

\medskip \noindent
\underline{Case 1.2:} Assume that $C$ has length $2\ell >4$. We know from Theorem~\ref{thm:LongCyclesConvex} that $C$ is gated. Let $a' \in C$ denote the projection of $x'$ on $C$. Observe that, because $a$ and $b$ lie on a geodesic from $x'$ to $y$, necessarily $a'$ lies between $b$ and $c$ on $C$. If $a'=c$, then by combining $d(x,a)=d(x,c)+\ell$ with
$$d(x,a)=d(x,x')+d(x',a)=d(x,x')+d(x',c)+\ell,$$
we deduce that $d(x,c')=d(x,x')+d(x',c)$. But then
$$d(x,y)=d(x,c)+d(c,y)= d(x,x')+d(x',c)+d(c,y),$$
which implies that $x' \in I(x,y)$, contradicting our assumptions. Therefore, we must have $a' \neq c$. As a consequence, $a'$ has a neighbour $y'$ lying between $a'$ and $c$ (with possibly $y'=c$). We have
$$\begin{array}{lcl} d(x,y') & = & d(x,c)+d(c,y') = d(x,a)- d(a,c)+d(c,y') = d(x,a)-d(a,y') \\ \\ & = & d(x,x')+d(x',a')+d(a',a) - d(a,y') = d(x,x')+d(x',a')-1 \\ \\ & = & d(x,x')+ d(x',y')-2 \end{array}$$
hence $x' \notin I(x,y')$ with $d(x',y') <d(x',y)$, contradicting our choice for $y$.

\medskip \noindent
\underline{Case 2:} Assume finally that $d(x,y)=d(x,a)$. By applying the triangle condition, we find a common neighbour $c$ of $y$ and $a$ that lies in $I(x,y) \cap I(x,a)$. Observe that $d(x',c) \geq d(x',y)-1=d(x',a)$, so
$$d(x,c)=d(x,a)-1=d(x,x')+d(x',a)-1 \leq d(x,x')+d(x',c)-1,$$
which implies that $x' \notin I(x,c)$. On the other hand, we know that $d(x',c) \geq d(x',y)-1$ and $d(x',c) \leq d(x',a)+1=d(x',y)$, so either $d(x',c)=d(x',y)$ and we apply Case 1 to get a contradiction, or $d(x',c)=d(x',y)-1$ and we already get a contradiction. 
\end{proof}

\noindent
Fix a graph $\Gamma$, a collection of groups $\mathcal{G}=\{G_u \mid u \in V(\Gamma)\}$ indexed by the vertices of $\Gamma$, and a labelling $\lambda : E(\Gamma) \to \mathbb{N}_{\geq 2}$ such that, for every edge $\{u,v\} \in E(\Gamma)$, if $\lambda(u,v) >2$ then $G_u,G_v$ are both cyclic of order two. We denote the associated periagroup $\Pi(\Gamma,\lambda,\mathcal{G})$ by $\Pi$. We want to prove Theorem~\ref{thm:IntroParabolic}, namely, we want to prove that the intersection between any two parabolic subgroups of $\Pi$ is again a parabolic subgroup. We begin by observing the following consequence of Proposition~\ref{prop:LocallyGated}. 

\begin{cor}\label{cor:ParabolicGated}
For every $\Xi \leq \Gamma$, the subgraph $\langle \Xi \rangle$ of $X:= \mathrm{Cayl}(\Pi,\mathcal{G})$ is gated.
\end{cor}

\begin{proof}
As a subgraph, $\langle \Xi \rangle$ is clearly connected, and it is locally gated as a consequence of the descriptions of $3$-cycles and convex even cycles respectively given by Lemma~\ref{lem:CliquePeria} and Claim~\ref{claim:ConvexDihedral} (which applies according to Proposition~\ref{prop:PeriagroupsAsRotation}). 
\end{proof}

\begin{proof}[Proof of Theorem~\ref{thm:IntroParabolic}.]
Let $\Phi,\Psi \leq \Gamma$ be two subgraphs and $g,h \in \Pi(\Gamma, \mathcal{G})$ two elements. According to Corollary~\ref{cor:ParabolicGated}, the subgraphs $g \langle \Phi \rangle$ and $h \langle \Psi \rangle$ are gated, so projections on them are well-defined. Because $g \langle \Phi \rangle g^{-1} \cap h \langle \Psi \rangle h^{-1}$ stabilises both $g \langle \Phi \rangle$ and $h \langle \Psi \rangle$, it has to stabilise the projection $P$ of $h \langle \Psi \rangle$ on $g \langle \Phi \rangle$. Let $\Xi$ denote the subgraph of $\Gamma$ given by generators labelling the edges of $P$. Assuming, up to translating by an element of $\Pi(\Gamma, \mathcal{G})$, that $P$ contains the vertex $1$, it follows that the intersection between our two parabolic subgroups lies in $\langle \Xi \rangle$. It remains to verify that $\langle \Xi \rangle$ stabilises both $g \langle \Phi \rangle$ and $h \langle \Psi \rangle$. So fix a generator $s \in \langle \Xi \rangle$. By definition of $\Xi$, there exists an edge in $g \langle \Phi \rangle$ labelled by $s$ that is the projection of an edge of $h \langle \Psi \rangle$. As a consequence of Corollary~\ref{cor:ProjLip}, these two edges belong to the same hyperplane, so $s$ belongs to the rotative-stabiliser of a hyperplane that intersects both $g \langle \Phi \rangle$ and $h \langle \Psi \rangle$. In order to conclude the proof, it suffices to notice that:

\begin{claim}
For every subgraph $\Delta \leq \Gamma$ and every clique $C \subset \langle \Delta \rangle$, $\mathrm{stab}(C)$ stabilises the subgraph $\langle \Delta \rangle$. 
\end{claim}

\noindent
We can write $C$ as $k \langle v \rangle$ for some element $k \in \langle \Delta \rangle$ and some vertex $v \in \Delta$. Then $\mathrm{stab}(C) = k \langle v \rangle k^{-1} \subset \langle \Delta \rangle$, so $\mathrm{stab}(C)$ stabilises the subgraph $\langle \Delta \rangle$. 
\end{proof}

\section{Concluding remarks and open questions}\label{section:Conclusion}

\noindent
In this article, we have introduced and motivated mediangle graphs as a natural geometry to consider in geometric group theory. This leads to many questions that remain to explore. We record some of these questions in this final section.

\paragraph{Cellular structure.} All the examples of mediangle graphs met in Section~\ref{section:Examples} can be naturally embedded with higher dimensional cell structures, typically contractible and often CAT(0). It is natural to look for a similar structure in mediangle graphs. In other words, we would like to identify specific subgraphs, the \emph{monoliths}, such that each monolith arises with a clear cell structure and such that filling in all the monoliths of a mediangle graph produces a cell structure with remarkable properties such as being contractible or admitting a CAT(0) metric.

\medskip \noindent
In Cayley graphs of periagroups, such a cell structure can be easily defined. Given a labelled graph $(\Gamma,\lambda)$ and a collection of groups $\mathcal{G}$ defining a periagroup $\Pi:= \Pi(\Gamma,\lambda, \mathcal{G})$, define a monolith as a subgraph of the form $g \langle \Lambda \rangle$ where $g \in \Pi$ is an element and where $\Lambda \leq \Gamma$ is a complete subgraph such that the Coxeter group associated to the vertices of $\Lambda$ with vertex-groups of order two is finite. A monolith is thus isomorphic to the product of a finite Coxeter graph with a (possibly infinite) complete graph. In particular, it can be realised as a convex polytope in a (possibly infinite-dimensional) Euclidean space. This cell structure naturally extends the known cases of Coxeter groups and graph products of groups, and it is reasonable to conjecture that it can also be endowed with a CAT(0) metric. This conjecture is also supported by \cite{Soergel}, which constructs CAT(0) cell complexes for Dyer groups in the same spirit.

\medskip \noindent
However, identifying monoliths in full generality is not so easy. Mediangle graphs are not built from obvious ``elementary blocks''.

\begin{question}
With a well-chosen definition of monoliths, is it always possible to endow a mediangle graph with a contractible or CAT(0) cell structure?
\end{question}

\noindent
More algebraically, one can ask whether a group acting geometrically on a mediangle graph is of type $F_\infty$, and of type $F$ when it is torsion-free.

\medskip \noindent
Here is a natural attempt for a definition of monoliths in full generality: in a mediangle graph, a monolith is a gated subgraph whose hyperplanes are all pairwise transverse. However, the structure of such monoliths is not clear. For instance, are they convex subgraphs in products of complete graphs with finite Coxeter graphs? A good indication that our guess is a good guess would be a positive answer to the following question:

\begin{question}
If a group acts on a mediangle graph with bounded orbits, does it stabilise a monolith?
\end{question}

\noindent
A positive answer to this question, combined with a contractible cell structure on monoliths, would exhibit mediangle graphs as a source of classifying spaces for proper actions.

\paragraph{Coxeter graphs.} As mentioned in the introduction, median graphs can be thought of as the combinatorial geometry of right-angled Coxeter groups in the sense that a graph is median if and only if it can be realised as a convex subgraph in the Cayley graph of some right-angled Coxeter group. Similarly, quasi-median graphs can be thought of as the combinatorial geometry of graph products of groups in the sense that a graph is quasi-median if and only if it can be realised as a gated subgraph in the Cayley graphs of some graph products of groups. So one naturally asks whether every mediangle graph can be realised as a gated subgraph in the Cayley graph of some periagroup. 

\medskip \noindent
Unfortunately, this not the case. For instance, consider the mediangle graph $X$ below. By considering the link of a vertex belonging to the $4$-cycle and to two $6$-cycles, one sees that there is essentially a unique candidate for a periagroup containing $X$ in its Cayley graph: the Coxeter group $(2,3,3)$. The corresponding Coxeter graph $Y$ is the one-skeleton of the truncated octahedron, and, even though $X$ can be realised as a subgraph in $Y$, such a subgraph cannot be convex. 
\begin{center}
\includegraphics[width=0.7\linewidth]{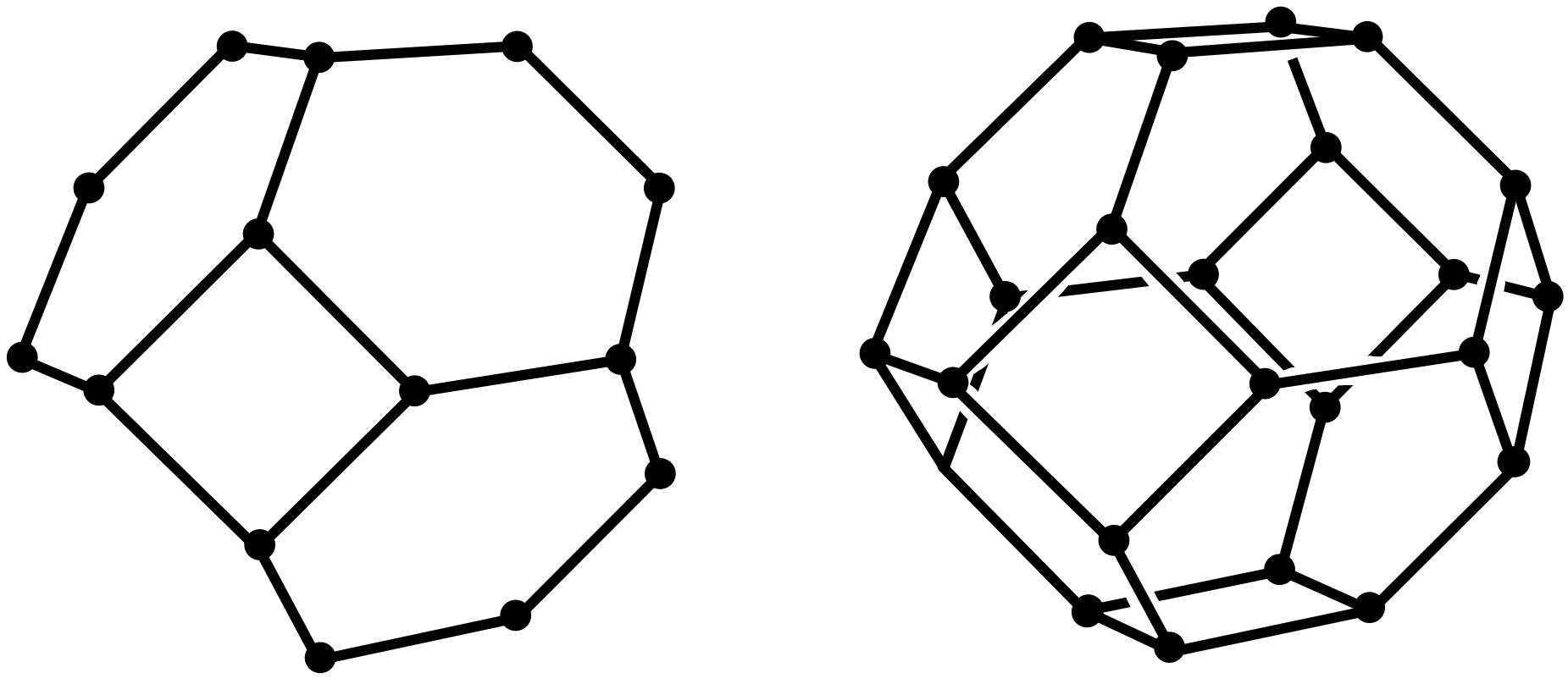}
\end{center}

\noindent
Consequently, there might exist a subfamily of mediangle graphs with a richer structure that remain to identify. 

\begin{question}\label{question:Coxeter}
Which mediangle graphs can be realised as gated subgraphs in Cayley graphs of periagroups? Which bipartite mediangle graphs can be realised as convex subgraphs in Coxeter graphs?
\end{question}

\noindent
A (too) particular case would be to ask for local reflection symmetries. More precisely, say that a graph $X$ is \emph{locally Coxeter} if, for any two adjacent vertices $x,y \in X$, there exists a bijection $\rho : \mathrm{link}(x) \to \mathrm{link}(y)$ such that $y$ is sent to $x$ and such that two neighbours $a,b \in \mathrm{link}(x)$ of $x$ span a convex cycle of length $\ell$ if and only if so do $\rho(a),\rho(b)$. Is it true that a bipartite mediangle graph that is locally Coxeter can be realised as a convex subgraph in a Coxeter graph? Is it isomorphic to a Coxeter graph?

\medskip \noindent
We emphasize that requiring that the subgraph in Question~\ref{question:Coxeter} is convex or gated is fundamental. Indeed, we already know that mediangle graphs embed isometrically in products of complete graphs, similarly to median and quasi-median graphs.

\begin{prop}
Every mediangle graph embeds isometrically in a product of (possibly infinitely many) complete graphs.
\end{prop}

\begin{proof}
Let $X$ be a mediangle graph. Fix a collection of cliques $\{C_i \mid i \in I\}$ such that every hyperplane contains exactly one clique in this collection. We know from Lemma~\ref{lem:CliquesGated} that cliques are gated, so there is a well-defined projection $\pi_i : X \to C_i$ for every $i \in I$. We claim that
$$\pi : \left\{ \begin{array}{ccc} X & \to & \prod\limits_{i \in I} C_i \\ x & \mapsto & (\pi_i(x))_{i \in I} \end{array} \right.$$
is an isometric embedding. Indeed, it follows from Corollary~\ref{cor:ProjLip} that, for every $i \in I$, two vertices have distinct projections on $C_i$ if and only if they are separated by the hyperplane containing $C_i$. Consequently, for any two vertices $x,y \in X$, the distance between $\pi(x)$ and $\pi(y)$ coincides with the number of hyperplanes separating $x,y$ in $X$, which also coincides with the distance between $x,y$ in $X$ according to Theorem~\ref{thm:BigHyp}.$(iii)$.
\end{proof}

\noindent
A more algebraic related question is to generalise the theory of special cube complexes introduced in \cite{MR2377497} in the spirit of its quasi-median generalisation \cite{SpecialBis}. In other words, is it possible to define \emph{special actions} on mediangle graphs so that every group admitting such an action embeds into a periagroup as a gated-cocompact subgroup? Or into another families of groups admitting mediangle Cayley graphs?

\medskip \noindent
In this direction, it is natural to associate to every hypergraph $\Gamma$ whose hyperedges are cyclically ordered the \emph{palindromic group}
$$\mathrm{Pal}(\Gamma):= \langle V(\Gamma) \mid u_1 \cdots u_n =u_n \cdots u_1, \ \{u_1 < \cdots < u_n \} \in E(\Gamma) \rangle.$$
When hyperedges have size two, one recovers right-angled Artin groups. Then, given a bipartite mediangle graph $X$, define its \emph{crossing hypergraph} $\Delta X$ as the hypergraph whose vertices are the hyperplanes of $X$ and whose hyperedges are given by the (cyclically ordered) hyperplanes crossing a given convex cycle. By extension to the (quasi-)median case, it is natural to ask whether the Cayley graph of $\mathrm{Pal}(\Delta X)$ is mediangle and whether $X$ embeds as a convex subgraph in it. Unfortunately, Cayley graphs of palindromic groups may not be mediangle.

\paragraph{Groups acting on mediangle graphs.} As mentioned in the introduction, median and quasi-median graphs, although tightly connected to right-angled Coxeter groups and graph products of groups, have interesting applications in many other families of groups. Therefore, it is natural to ask whether, similarly, there exist interesting families of groups that are not obviously related to periagroups but that act non-trivially on mediangle graphs. We did not address this question in this article, which remains to be explored, but we would like to mention at least one intersection question in this direction.

\begin{question}\label{question:Palindrome}
Given a hypergraph $\Gamma$ whose hyperedges are totally ordered, when is $\mathrm{Cayl}(\mathrm{Pal}(\Gamma) , V(\Gamma))$ mediangle?
\end{question}

\noindent
There are necessary restrictions on $\Gamma$. For instance, the intersection between two distinct hyperedges cannot contain more than a single vertex; if a vertex is $2$-adjacent (i.e.\ belongs to a hyperedge of size $2$) with at least two vertices of hyperedge, then it must be $2$-adjacent to all the vertices of this hyperedge. However, these conditions are not sufficient, and there is not obvious candidate for an exhaustive list of conditions. 

\noindent
The simplest palindromic group is $P_n:= \langle x_1, \ldots, x_n \mid x_1 \cdots x_n = x_n \cdots x_1 \rangle$. Observe that this presentation satisfies the small cancellation condition $C(2n)-T(4)$, so it follows from Proposition~\ref{prop:SmallCancellation} that $\mathrm{Cayl}(P_n,\{x_1, \ldots, x_n\})$ is a mediangle graph. Many other palindromic groups have their presentations that satisfy the small cancellation condition C(4)-T(4), and so are covered by Proposition~\ref{prop:SmallCancellation}, but this represents only a small portion of palindromic groups. 

\medskip \noindent
A more geometric problem related to Question~\ref{question:Palindrome} is:

\begin{question}
Which hypergraphs with cyclically ordered hyperedges are crossing hypergraphs of mediangle graphs?
\end{question}

\noindent
For instance, a hypergraph with vertex-set $\{a,b,c,d\}$ with hyperedges $\{a,b,c\}$ and $\{b,c,d\}$ is not the crossing hypergraph of a mediangle graph. This contrasts with crossing graphs of (quasi-)median graphs, which can be arbitrary graphs.

\paragraph{Roller boundary.} A median graph has a natural boundary, defined in terms of \emph{ultrafilters} or \emph{orientations}, with the convenient property that it is a graph and that all of its components are again median graphs \cite{Roller}. Geometrically, this \emph{Roller graph} can also be defined in terms of infinite geodesic rays \cite{MR4071367}. A similar construction for Coxeter groups is possible in terms of infinite words (which amounts to considering infinite geodesic rays) as illustrated in \cite{MR3403997}. It would be interesting to unify this construction in the framework of mediangle graphs. 

\medskip \noindent
An \emph{orientation} of a mediangle graph $X$ is a map $\sigma$ that chooses for each hyperplane one of its sectors and that satisfies $\bigcap_{1 \leq i \leq n} \sigma(J_i) \neq \emptyset$ for all hyperplanes $J_1, \ldots, J_n $. For instance, if $x \in X$ is a vertex, the map $\sigma_x$ that maps each hyperplane to its sector containing $x$ is an orientation, referred to as a \emph{principal orientation}. Let $\bar{X}$ denote the graph whose vertices are the orientations and whose edges connect two orientations whenever they differ on a single hyperplane. The \emph{Roller boundary} $\mathfrak{R}X$ of $X$ is the $\bar{X}\backslash X$ where we identify $X$ with its image in $\bar{X}$ under $x \mapsto \sigma_x$ (which induces a graph embedding whose image is a whole connected component of $\bar{X}$). 

\medskip \noindent
A more geometric description of $\mathfrak{R}X$ would be to fix a vertex $o \in X$, to define the vertices of $\mathfrak{R}X$ as the infinite geodesic rays starting from $o$ modulo the relation that identifies two rays if they cross the same hyperplane, and to define the edges of $\mathfrak{R}X$ by imposing that any two rays that cross the same hyperplanes but one are adjacent. 

\begin{question}
Given a mediangle graph, is a connected component of its Roller boundary mediangle?
\end{question}

\noindent
For median graphs, the Roller boundary play a key role in the proofs of the rigidity rank-one conjecture and the Tits alternative \cite{MR2827012}. Therefore, an algebraic application of the previous geometric considerations would be to prove similar results for groups acting on mediangle graphs. 

\medskip \noindent
We plan to investigate Roller boundaries of mediangle graphs in a forthcoming work.

\paragraph{Local-to-global characterisation.} (Quasi-)median graphs can be easily recognised thanks to local criteria. For mediangle graphs, this is more delicate. 

\medskip \noindent
\begin{minipage}{0.3\linewidth}
\includegraphics[width=0.7\linewidth]{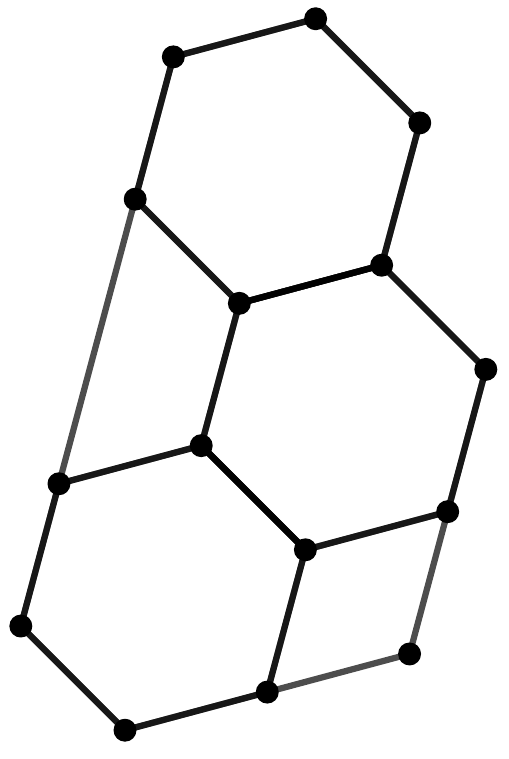}
\end{minipage}
\begin{minipage}{0.68\linewidth}
For instance, the graph on the left has an obvious structure of simply connected polygonal complex, links are isomorphic to links of vertices in Coxeter polytopes (so do not provide local obstruction to being mediangle), but the graph is not mediangle.

\medskip \noindent
Nevertheless, a criterion in the spirit of \cite[Theorem~3.1]{MR4223822}, which characterises weakly modular graphs, is still possible.
\end{minipage}

\begin{question}
Is there an easy way to recognise mediangle graphs?
\end{question}

\paragraph{Wallspaces with angles.} Cubulating spaces with walls is fundamental in the study of groups acting on median graphs, both for producing examples and to study such actions in full generality. Does there exist a similar construction for mediangle graphs? Here, in addition to hyperplanes, we have well-defined angles between transverse hyperplanes. Is there a reasonable definition of \emph{wallspaces with angles} such that a mediangle graph can be associated to every such space?

\paragraph{Combination theorems.} In \cite{QM}, we introduce a general framework that allows us to prove combination theorems for groups acting on quasi-median graphs, including graph products of groups. (See \cite{TSG} for a gentle introduction to the main ideas of the construction.) This approach can be used to show that many properties are stable under graph products of groups (over finite graphs) including: acting (metrically) properly on a median graph, being a-T-menable, being CAT(0) \cite{QM}; being Helly \cite{Helly}; being coarse median, and being strongly shortcut \cite{AutGP}. It would be natural to extend \cite{QM} to mediangle graphs, and in particular to periagroups.

\addcontentsline{toc}{section}{References}

\bibliographystyle{alpha}
{\footnotesize\bibliography{Mediangle}}

@book {Davis,
    AUTHOR = {Davis, M.},
     TITLE = {The geometry and topology of {C}oxeter groups},
    SERIES = {London Mathematical Society Monographs Series},
    VOLUME = {32},
 PUBLISHER = {Princeton University Press, Princeton, NJ},
      YEAR = {2008},
     PAGES = {xvi+584},
      ISBN = {978-0-691-13138-2; 0-691-13138-4},
   MRCLASS = {20F55 (05B45 05C25 51-02 57M07)},
  MRNUMBER = {2360474},
MRREVIEWER = {Ralf Koehl},
}

@article {MR4223822,
    AUTHOR = {Chalopin, J. and Chepoi, V. and Hirai, H. and
              Osajda, D.},
     TITLE = {Weakly modular graphs and nonpositive curvature},
   JOURNAL = {Mem. Amer. Math. Soc.},
  FJOURNAL = {Memoirs of the American Mathematical Society},
    VOLUME = {268},
      YEAR = {2020},
    NUMBER = {1309},
     PAGES = {vi+159},
      ISSN = {0065-9266},
      ISBN = {978-1-4704-4362-7; 978-1-4704-6349-6},
   MRCLASS = {05C12 (05C75 05E45 20F67 51K05)},
  MRNUMBER = {4223822},
       DOI = {10.1090/memo/1309},
       URL = {https://doi-org.ezpum.biu-montpellier.fr/10.1090/memo/1309},
}

@article {MR4065838,
    AUTHOR = {Chepoi, V. and Knauer, K. and Marc, T.},
     TITLE = {Hypercellular graphs: partial cubes without {$Q_3^-$} as
              partial cube minor},
   JOURNAL = {Discrete Math.},
  FJOURNAL = {Discrete Mathematics},
    VOLUME = {343},
      YEAR = {2020},
    NUMBER = {4},
     PAGES = {111678, 28},
      ISSN = {0012-365X},
   MRCLASS = {05C62 (05C40)},
  MRNUMBER = {4065838},
MRREVIEWER = {Marko Jakovac},
       DOI = {10.1016/j.disc.2019.111678},
       URL = {https://doi-org.ezpum.biu-montpellier.fr/10.1016/j.disc.2019.111678},
}

@incollection {MR1379365,
    AUTHOR = {Bandelt, H.-J. and Chepoi, V.},
     TITLE = {Cellular bipartite graphs},
      NOTE = {Discrete metric spaces (Bielefeld, 1994)},
   JOURNAL = {European J. Combin.},
  FJOURNAL = {European Journal of Combinatorics},
    VOLUME = {17},
      YEAR = {1996},
    NUMBER = {2-3},
     PAGES = {121--134},
      ISSN = {0195-6698},
   MRCLASS = {05C75 (05C12 05C85)},
  MRNUMBER = {1379365},
MRREVIEWER = {Ortrud R. Oellermann},
       DOI = {10.1006/eujc.1996.0011},
       URL = {https://doi-org.ezpum.biu-montpellier.fr/10.1006/eujc.1996.0011},
}

@article{TSG,
     author = {Genevois, A.},
     title = {Groups acting on quasi-median graphs. {An} introduction},
     journal = {S\'eminaire de th\'eorie spectrale et g\'eom\'etrie},
     pages = {43--68},
     publisher = {Institut Fourier},
     address = {Grenoble},
     volume = {35},
     year = {2017-2019},
     doi = {10.5802/tsg.363},
     language = {en},
     url = {https://tsg.centre-mersenne.org/articles/10.5802/tsg.363/}
}

@article{Helly,
     author = {Chalopin, J. and Chepoi, V. and Genevois, A. and Hirai, H. and Osajda, D.},
     title = {Helly groups},
     journal = {arxiv:2002.06895},
     year = {2020},
}

@article {MR1748966,
    AUTHOR = {Chepoi, V.},
     TITLE = {Graphs of some {${\rm CAT}(0)$} complexes},
   JOURNAL = {Adv. in Appl. Math.},
  FJOURNAL = {Advances in Applied Mathematics},
    VOLUME = {24},
      YEAR = {2000},
    NUMBER = {2},
     PAGES = {125--179},
      ISSN = {0196-8858},
   MRCLASS = {57M07 (05C10 20F67)},
  MRNUMBER = {1748966},
       DOI = {10.1006/aama.1999.0677},
       URL = {https://doi-org.ezpum.biu-montpellier.fr/10.1006/aama.1999.0677},
}

@article {MR2397349,
    AUTHOR = {Bandelt, H.-J. and Chepoi, V.},
     TITLE = {The algebra of metric betweenness. {II}. {G}eometry and
              equational characterization of weakly median graphs},
   JOURNAL = {European J. Combin.},
  FJOURNAL = {European Journal of Combinatorics},
    VOLUME = {29},
      YEAR = {2008},
    NUMBER = {3},
     PAGES = {676--700},
      ISSN = {0195-6698},
   MRCLASS = {05C12 (05C75 06B05)},
  MRNUMBER = {2397349},
MRREVIEWER = {Bo\v{s}tjan Bre\v{s}ar},
       DOI = {10.1016/j.ejc.2007.03.003},
       URL = {https://doi-org.ezpum.biu-montpellier.fr/10.1016/j.ejc.2007.03.003},
}

@book {MR605838,
    AUTHOR = {Mulder, H.},
     TITLE = {The interval function of a graph},
    SERIES = {Mathematical Centre Tracts},
    VOLUME = {132},
 PUBLISHER = {Mathematisch Centrum, Amsterdam},
      YEAR = {1980},
     PAGES = {iii+191},
      ISBN = {90-6196-208-0},
   MRCLASS = {05C99 (05C65 06B99)},
  MRNUMBER = {605838},
MRREVIEWER = {M. E. Watkins},
}

@article {MR1888425,
    AUTHOR = {McCammond, J. and Wise, D.},
     TITLE = {Fans and ladders in small cancellation theory},
   JOURNAL = {Proc. London Math. Soc. (3)},
  FJOURNAL = {Proceedings of the London Mathematical Society. Third Series},
    VOLUME = {84},
      YEAR = {2002},
    NUMBER = {3},
     PAGES = {599--644},
      ISSN = {0024-6115},
   MRCLASS = {20F06},
  MRNUMBER = {1888425},
MRREVIEWER = {Vladimir N. Bezverkhni\u{\i}},
       DOI = {10.1112/S0024611502013424},
       URL = {https://doi-org.ezpum.biu-montpellier.fr/10.1112/S0024611502013424},
}

@article {MR2377497,
    AUTHOR = {Haglund, F. and Wise, D.},
     TITLE = {Special cube complexes},
   JOURNAL = {Geom. Funct. Anal.},
  FJOURNAL = {Geometric and Functional Analysis},
    VOLUME = {17},
      YEAR = {2008},
    NUMBER = {5},
     PAGES = {1551--1620},
      ISSN = {1016-443X},
   MRCLASS = {20F36 (20F55 20F67)},
  MRNUMBER = {2377497},
MRREVIEWER = {Patrick Bahls},
       DOI = {10.1007/s00039-007-0629-4},
       URL = {https://doi-org.ezpum.biu-montpellier.fr/10.1007/s00039-007-0629-4},
}

@article {MR2827012,
    AUTHOR = {Caprace, P.-E. and Sageev, M.},
     TITLE = {Rank rigidity for {CAT}(0) cube complexes},
   JOURNAL = {Geom. Funct. Anal.},
  FJOURNAL = {Geometric and Functional Analysis},
    VOLUME = {21},
      YEAR = {2011},
    NUMBER = {4},
     PAGES = {851--891},
      ISSN = {1016-443X},
   MRCLASS = {20F65 (20F67 53C24)},
  MRNUMBER = {2827012},
MRREVIEWER = {Tetsu Toyoda},
       DOI = {10.1007/s00039-011-0126-7},
       URL = {https://doi-org.ezpum.biu-montpellier.fr/10.1007/s00039-011-0126-7},
}

@article {QM,
    AUTHOR = {Genevois, A.},
     TITLE = {Cubical-like geometry of quasi-median graphs and applications to geometric group theory},
   JOURNAL = {PhD Thesis, arXiv:1712.01618},
	YEAR = {2017},
}

@article {SpecialBis,
    AUTHOR = {Genevois, A.},
     TITLE = {Special cube complexes revisited: a quasi-median generalisation},
   JOURNAL = {arXiv:2002.01670},
	YEAR = {2020},
}

@article {Roller,
    AUTHOR = {Roller, M.},
     TITLE = {Pocsets, median algebras and group actions: an extended study of {D}unwoody's construction and {S}ageev's theorem},
   JOURNAL = {dissertation},
	YEAR = {1998},
}

@article {MR3403997,
    AUTHOR = {Lam, T. and Thomas, A.},
     TITLE = {Infinite reduced words and the {T}its boundary of a {C}oxeter
              group},
   JOURNAL = {Int. Math. Res. Not. IMRN},
  FJOURNAL = {International Mathematics Research Notices. IMRN},
      YEAR = {2015},
    NUMBER = {17},
     PAGES = {7690--7733},
      ISSN = {1073-7928},
   MRCLASS = {20F55 (20F65 20F69)},
  MRNUMBER = {3403997},
MRREVIEWER = {Stefan Witzel},
       DOI = {10.1093/imrn/rnu182},
       URL = {https://doi-org.ezpum.biu-montpellier.fr/10.1093/imrn/rnu182},
}

@article {MR4071367,
    AUTHOR = {Genevois, A.},
     TITLE = {Contracting isometries of {${\rm CAT}(0)$} cube complexes and
              acylindrical hyperbolicity of diagram groups},
   JOURNAL = {Algebr. Geom. Topol.},
  FJOURNAL = {Algebraic \& Geometric Topology},
    VOLUME = {20},
      YEAR = {2020},
    NUMBER = {1},
     PAGES = {49--134},
      ISSN = {1472-2747},
   MRCLASS = {20F65 (20F67)},
  MRNUMBER = {4071367},
MRREVIEWER = {Michael L. Mihalik},
       DOI = {10.2140/agt.2020.20.49},
       URL = {https://doi-org.ezpum.biu-montpellier.fr/10.2140/agt.2020.20.49},
}

@article {AutGP,
    AUTHOR = {Genevois, A.},
     TITLE = {Automorphisms of graph products of groups and acylindrical hyperbolicity},
   JOURNAL = {arXiv:1807.00622},
	YEAR = {2018},
}

@article {Wreath,
    AUTHOR = {Genevois, A. and Tessera, R.},
     TITLE = {Asymptotic geometry of lamplighters over one-ended groups},
   JOURNAL = {arXiv:2105.04878},
	YEAR = {2021},
}

@article {GreenGP,
    AUTHOR = {Green, E.},
     TITLE = {Graph products of groups},
   JOURNAL = {PhD Thesis},
	YEAR = {1990},
}

@article {MR4295519,
    AUTHOR = {Genevois, A. and Martin, A.},
     TITLE = {Automorphisms of graph products of groups from a geometric
              perspective},
   JOURNAL = {Proc. Lond. Math. Soc. (3)},
  FJOURNAL = {Proceedings of the London Mathematical Society. Third Series},
    VOLUME = {119},
      YEAR = {2019},
    NUMBER = {6},
     PAGES = {1745--1779},
      ISSN = {0024-6115},
   MRCLASS = {20F65 (20F28)},
  MRNUMBER = {4295519},
       DOI = {10.1112/plms.12282},
       URL = {https://doi-org.ezpum.biu-montpellier.fr/10.1112/plms.12282},
}

@article {MR4033512,
    AUTHOR = {Genevois, A.},
     TITLE = {Embeddings into {T}hompson's groups from quasi-median
              geometry},
   JOURNAL = {Groups Geom. Dyn.},
  FJOURNAL = {Groups, Geometry, and Dynamics},
    VOLUME = {13},
      YEAR = {2019},
    NUMBER = {4},
     PAGES = {1457--1510},
      ISSN = {1661-7207},
   MRCLASS = {20F65 (05C25)},
  MRNUMBER = {4033512},
MRREVIEWER = {Bruno P. Zimmermann},
       DOI = {10.4171/ggd/530},
       URL = {https://doi-org.ezpum.biu-montpellier.fr/10.4171/ggd/530},
}

@book {MR1812024,
    AUTHOR = {Lyndon, R. and Schupp, P.},
     TITLE = {Combinatorial group theory},
    SERIES = {Classics in Mathematics},
      NOTE = {Reprint of the 1977 edition},
 PUBLISHER = {Springer-Verlag, Berlin},
      YEAR = {2001},
     PAGES = {xiv+339},
      ISBN = {3-540-41158-5},
   MRCLASS = {20Fxx (20Exx 57M07)},
  MRNUMBER = {1812024},
       DOI = {10.1007/978-3-642-61896-3},
       URL = {https://doi-org.ezpum.biu-montpellier.fr/10.1007/978-3-642-61896-3},
}

@incollection {MR0254129,
    AUTHOR = {Tits, J.},
     TITLE = {Le probl\`eme des mots dans les groupes de {C}oxeter},
 BOOKTITLE = {Symposia {M}athematica ({INDAM}, {R}ome, 1967/68), {V}ol. 1},
     PAGES = {175--185},
 PUBLISHER = {Academic Press, London},
      YEAR = {1969},
   MRCLASS = {20.10},
  MRNUMBER = {0254129},
MRREVIEWER = {J. E. Humphreys},
}

@article {Excentric,
    AUTHOR = {Genevois, A.},
     TITLE = {Quasi-isometrically rigid subgroups in right-angled Coxeter groups},
   JOURNAL = {arxiv:1909.04318, to appear in Algebraic \& Geometric Topology},
      YEAR = {2019},
}

@article {RAAGRACG,
    AUTHOR = {Dani, P. and Levcovitz, I.},
     TITLE = {Right-angled {A}rtin subgroups of right-angled {C}oxeter and {A}rtin groups},
   JOURNAL = {arxiv:2003.05531},
      YEAR = {2020},
}

@article {Soergel,
    AUTHOR = {Soergel, M.},
     TITLE = {A generalization of the {D}avis-{M}oussong complex for {D}yer groups},
   JOURNAL = {arxiv:2212.03017},
      YEAR = {2022},
}

@article {MR2466021,
    AUTHOR = {Krammer, D.},
     TITLE = {The conjugacy problem for {C}oxeter groups},
   JOURNAL = {Groups Geom. Dyn.},
  FJOURNAL = {Groups, Geometry, and Dynamics},
    VOLUME = {3},
      YEAR = {2009},
    NUMBER = {1},
     PAGES = {71--171},
      ISSN = {1661-7207},
   MRCLASS = {20F55 (20F10)},
  MRNUMBER = {2466021},
MRREVIEWER = {Robert H. Gilman},
       DOI = {10.4171/GGD/52},
       URL = {https://doi-org.ezpum.biu-montpellier.fr/10.4171/GGD/52},
}

@article {MR3365774,
    AUTHOR = {Antol\'{\i}n, Y. and Minasyan, A.},
     TITLE = {Tits alternatives for graph products},
   JOURNAL = {J. Reine Angew. Math.},
  FJOURNAL = {Journal f\"{u}r die Reine und Angewandte Mathematik. [Crelle's
              Journal]},
    VOLUME = {704},
      YEAR = {2015},
     PAGES = {55--83},
      ISSN = {0075-4102},
   MRCLASS = {20F65 (03E10)},
  MRNUMBER = {3365774},
MRREVIEWER = {J\"{o}rg Lehnert},
       DOI = {10.1515/crelle-2013-0062},
       URL = {https://doi-org.ezpum.biu-montpellier.fr/10.1515/crelle-2013-0062},
}

@article {MR2371978,
    AUTHOR = {Duncan, A. and Kazachkov, I. and Remeslennikov, V.},
     TITLE = {Parabolic and quasiparabolic subgroups of free partially
              commutative groups},
   JOURNAL = {J. Algebra},
  FJOURNAL = {Journal of Algebra},
    VOLUME = {318},
      YEAR = {2007},
    NUMBER = {2},
     PAGES = {918--932},
      ISSN = {0021-8693},
   MRCLASS = {20F99},
  MRNUMBER = {2371978},
       DOI = {10.1016/j.jalgebra.2007.08.032},
       URL = {https://doi-org.ezpum.biu-montpellier.fr/10.1016/j.jalgebra.2007.08.032},
}

@article {MR1023969,
    AUTHOR = {Deodhar, V.},
     TITLE = {A note on subgroups generated by reflections in {C}oxeter
              groups},
   JOURNAL = {Arch. Math. (Basel)},
  FJOURNAL = {Archiv der Mathematik},
    VOLUME = {53},
      YEAR = {1989},
    NUMBER = {6},
     PAGES = {543--546},
      ISSN = {0003-889X},
   MRCLASS = {20H15 (20F55)},
  MRNUMBER = {1023969},
MRREVIEWER = {Fran\c{c}ois Digne},
       DOI = {10.1007/BF01199813},
       URL = {https://doi-org.ezpum.biu-montpellier.fr/10.1007/BF01199813},
}

@article {MR1076077,
    AUTHOR = {Dyer, M.},
     TITLE = {Reflection subgroups of {C}oxeter systems},
   JOURNAL = {J. Algebra},
  FJOURNAL = {Journal of Algebra},
    VOLUME = {135},
      YEAR = {1990},
    NUMBER = {1},
     PAGES = {57--73},
      ISSN = {0021-8693},
   MRCLASS = {20F55 (17B20 20H15)},
  MRNUMBER = {1076077},
MRREVIEWER = {Robert B\'{e}dard},
       DOI = {10.1016/0021-8693(90)90149-I},
       URL = {https://doi-org.ezpum.biu-montpellier.fr/10.1016/0021-8693(90)90149-I},
}

\Address

%

\end{document}